\documentclass[11pt,leqno]{amsart}
\usepackage{hyperref}
\usepackage{amssymb, amsmath, amsthm, amsfonts}
\usepackage{verbatim}
\usepackage{bbm}
\usepackage{enumerate}
\usepackage{dsfont}
\usepackage[mathscr]{eucal}
\usepackage{tikz}
\usepackage[normalem]{ulem}

\usepackage[lmargin=1.5in,rmargin=1.5in]{geometry}
\DeclareFontFamily{U}{mathx}{\hyphenchar\font45}
\DeclareFontShape{U}{mathx}{m}{n}{
<5> <6> <7> <8> <9> <10>
<10.95> <12> <14.4> <17.28> <20.74> <24.88>
mathx10
}{}
\DeclareSymbolFont{mathx}{U}{mathx}{m}{n}
\DeclareFontSubstitution{U}{mathx}{m}{n}
\DeclareMathAccent{\widecheck}{0}{mathx}{"71}
\DeclareMathAccent{\wideparen}{0}{mathx}{"75}

\usetikzlibrary{patterns}

\usepackage{todonotes}
\usepackage{comment}
\usepackage{color}
\definecolor{gray}{gray}{0.5}

\usepackage{tikz}

\newcommand{\trapez}{
\begin{tikzpicture} [scale=.23]
\draw (0,0) -- (.5,1);
\draw (0,0) -- (2,0);
\draw (1.5,1) -- (2,0);
\draw (.5,1) -- (1.5,1);
\end{tikzpicture}
}

\newcommand{\hypo}{{\mathrm{Hyp}(p_\circ,r_\circ)}}
\newcommand{\rstar}{{r_*(p,p_\circ,r_\circ)}}

\newcommand{\vertiii}[1]{{\left\vert\kern-0.25ex\left\vert\kern-0.25ex\left\vert #1
\right\vert\kern-0.25ex\right\vert\kern-0.25ex\right\vert}}

\newcommand{\VV}{{\mathrm{VBR}}}
\newcommand{\Xspace}{\mathcal V}

\newcommand{\ud}{\, \mathrm{d}}

\def\lc{\lesssim}

\def\sA{\mathscr{A}}

\def\eps{\varepsilon}
\def\bbone{{\mathbbm 1}}

\newcommand{\ci}[1]{_{{}_{\!\scriptstyle{#1}}}}

\newcommand{\Be}{\begin{equation}}
\newcommand{\Ee}{\end{equation}}

\newcommand{\Bm}{\begin{multline}}
\newcommand{\Em}{\end{multline}}

\newcommand{\nc}{n_\circ}
\newcommand{\RH}{\mathrm{RH}}

\def\intslash{\rlap{\kern .32em $\mspace {.5mu}\backslash$ }\int}
\def\qsl{{\rlap{\kern .32em $\mspace {.5mu}\backslash$ }\int_{Q_x}}}

\def\F{\mathcal F}

\def\lc{\lesssim}

\def\emph#1{{\it #1 }}
\def\diam{{\,\text{\rm diam}}}

\def\ga{\gamma}

\def\cf{{\it cf}}

\def\dist{{\mathrm{dist}}}

\def\supp{{\mathrm{supp}\,}}

\def\inn#1#2{\langle#1,#2\rangle}
\def\biginn#1#2{\big\langle#1,#2\big\rangle}
\def\Biginn#1#2{\Big\langle#1,#2\Big\rangle}
\def\jp#1{{\langle#1\rangle}}

\def\ga{\gamma} 
\def\eps{\varepsilon}
\def\ep{\epsilon}
\def\ka{\kappa}
\def\la{\lambda} \def\La{\Lambda}

\def\om{\omega} \def\Om{\Omega}
\def\vth{\vartheta}

\def\fD{{\mathfrak {D}}}
\def\fE{{\mathfrak {E}}}

\def\fQ{{\mathfrak {Q}}}
\def\fR{{\mathfrak {R}}}
\def\fS{{\mathfrak {S}}}

\def\fW{{\mathfrak {W}}}

\def\fn{{\mathfrak {n}}}

\def\bbC{{\mathbb {C}}}

\def\bbN{{\mathbb {N}}}

\def\bbR{{\mathbb {R}}}

\def\bbU{{\mathbb {U}}}
\def\bbV{{\mathbb {V}}}

\def\bbZ{{\mathbb {Z}}}

\def\R{{\mathbb {R}}}
\def\N{{\mathbb {N}}}

\def\cD{{\mathcal {D}}}

\def\cF{{\mathcal {F}}}

\def\cI{{\mathcal {I}}}

\def\cR{{\mathcal {R}}}

\def\cT{{\mathcal {T}}}

\def\cY{{\mathcal {Y}}}

\def\emph#1{{\it #1}}
\def\textbf#1{{\bf #1}}
\def\Sp{{\mathrm{Sp}}}

\def\beq{\begin{equation}}
\def\endeq{\end{equation}}

\def\bs{\begin{split}}
\def\es{\end{split}}

\theoremstyle{plain}
\newtheorem{thm}{Theorem}[section]
\newtheorem{prop}[thm]{Proposition}

\newtheorem{lemma}[thm]{Lemma}
\newtheorem{cor}[thm]{Corollary}
\newtheorem{definition}[thm]{Definition}

\newtheorem*{thm*}{Theorem}
\newtheorem*{conj*}{Conjecture}
\newtheorem*{openproblem*}{Open Problem}

\newtheorem*{hypothesis*}{Hypothesis}

\theoremstyle{remark}
\newtheorem{rem}[thm]{Remark}

\numberwithin{equation}{section}

\usepackage[utf8]{inputenc}

\usepackage[explicit,indentafter]{titlesec}
\titleformat{\section}{\centering\normalfont\scshape}{\thesection.}{.5em}{#1}
\titleformat{\subsection}[runin]{\normalfont\itshape}{\textnormal{\thesubsection.}}{.5em}{#1.}
\titleformat{\subsubsection}[runin]{\normalfont\itshape}{\thesubsubsection.}{.5em}{#1.}
\titlespacing{\section}{0em}{1em}{0.5em}
\titlespacing{\subsection}{0em}{.5em}{0.5em}


\begin{document}
\title[Sharp sparse bounds for Bochner-Riesz operators]{Bochner-Riesz means at the critical index:
\\ Weighted and sparse bounds}
\author[D. Beltran \ \ \ \ \ \ \ J. Roos \ \ \ \ \ \ \ A. Seeger ] {David Beltran \ \ \ \ Joris Roos \ \ \ \ Andreas Seeger }

\address{David Beltran: Departament d'Anàlisi Matemàtica, Universitat de València, Dr. Moliner 50, 46100 Burjassot, Spain}

\email{david.beltran@uv.es}

\address{Joris Roos: Department of Mathematics and Statistics, University of Massachusetts Lowell, Lowell, MA 01854, USA}
\email{joris\_roos@uml.edu}

\address{Andreas Seeger: Department of Mathematics, University of Wisconsin-Madison, 480 Lincoln Dr, Madison, WI-53706, USA}
\email{seeger@math.wisc.edu}

\begin{abstract}
We consider Bochner-Riesz means on weighted $L^p$ spaces, at the critical index $\lambda(p)=d(\frac 1p-\frac 12)-\frac 12$. For every $A_1$-weight we obtain an extension of Vargas' weak type $(1,1)$ inequality in some range of $p>1$.
To prove this result we establish new endpoint results for sparse domination.
These are almost optimal
in dimension
$d= 2$; partial results as well as conditional results are proved in higher dimensions.
For the means of index $\la_*=
\frac{d-1}{2d+2}$ we prove fully optimal sparse bounds.
\end{abstract}

\date{\today}

\maketitle

\section{Introduction}
Let $\Omega$ be a convex open subset of $\bbR^d$, $d\ge 2$, containing the origin. We assume that $\Omega$ has $C^\infty$-boundary with nonvanishing Gaussian curvature. Let \[\rho(\xi):= \inf\{t>0: \xi/t\in \Omega\}\] be the Minkowski functional of $\Omega$. Then
$\rho\in C^\infty(\bbR^d\setminus\{0\})$, $\rho$ is homogeneous of degree $1$, $\rho(\xi)>0$ for $\xi\neq 0$ and $\rho(\xi)=1$ on the boundary $\partial \Omega$.
Let $a>0$.
Given $\lambda >0$, we define the {\it Riesz means} of index $\lambda$ of the inverse Fourier integral by
\[ \cR^\la_{a,t} f(x):= \frac{1}{(2\pi)^{d}}
\int \Big(1- \frac{\rho(\xi)^a}{t^a} \Big)_+^\la \widehat f(\xi) e^{i \inn x\xi} \ud\xi, \]
where $\widehat f(\xi)=\int f(y)e^{-i\inn{y}{\xi} } \ud y$ denotes the Fourier transform of a Schwartz function $f$ on $\bbR^d$ and $s_+:=\max\{s,0\}$.
The case of $\Omega=\{\xi:|\xi|\le 1\} $ yields $\rho(\xi)=|\xi|$; in this case the means with $a=1$ are the {\it classical radial Riesz means} of index $\la$ while the case $a=2$ corresponds to the {\it Bochner-Riesz means} of index $\la$.

Given $1 \leq p < \frac{2d}{d+1}$, the value
\Be \label{eq:critical-index}\la(p):= d\Big(\frac 1p-\frac 12\Big)-\frac 12\Ee
is referred to as the \textit{critical index}, and it is conjectured that in this range the operators $\cR^{\la(p)}_{a,t}$ are of weak type $(p,p)$. The case $p=1$, corresponding to the index $\lambda(1)=\frac{d-1}{2}$, was first proved by Christ \cite{Christ-rough}
and later substantially extended by Vargas \cite{Vargas96} who proved an $L^1(w)\to L^{1,\infty}(w)$ result for all $A_1$ weights $w$, that is, for all $w \in L^1_{\mathrm{loc}}(\R^d)$ satisfying the pointwise inequality $Mw\lc w$, where $M$ denotes the Hardy-Littlewood maximal operator.
Sharp weak type endpoint results for $p>1$
were proved by Christ \cite{Christ-BR87} in the range $1< p< \frac{2(d+1)}{d+3}$, by Tao for $p=\frac{2(d+1)}{d+3}$, and complete results in two dimensions were obtained by one of the authors in \cite{seeger-BRwt}. Later, Tao \cite{Tao-Indiana1998} showed that for $1<p<\frac{2d}{d+1}$ the weak type endpoint estimates follow from the corresponding strong type results for all $\la>\la(p)$. For $d=2$ these are well-known and due to Carleson-Sj\"olin \cite{CarlesonSjolin}, allowing to recover the weak-type results from \cite{seeger-BRwt}. In higher dimensions many sharp partial results for the strong type estimate have been proved; see
\cite{fefferman69, LeeSanghyk2004, LeeSanghyk2018, Wu-JdA, GuoWangZhang2022}
and the references in those papers.

The goal of this paper is to establish new estimates for the operators $\cR_{a,t}^{\lambda(p)}$ whenever $1 < p < \frac{2d}{d+1}$.

\subsection{Weighted estimates}
We will be concerned with weights in the Muckenhoupt $A_s$ classes and the reverse Hölder classes $\RH_\sigma$; see \S\ref{sec:weighted} for the precise definitions. By testing against Schwartz functions it is easy to see that for $p<\frac{2d}{d+1}$ the operators $\cR_{a,t}^{\lambda(p)}$ fail to satisfy weighted weak-type $(p,p)$ estimates for the power weights $|x|^\eps$ for any $\eps>0$. This rules out, in particular, the Muckenhoupt $A_s$ classes for any $s>1$ (which can also be ruled out by the weak-type version of Rubio de Francia's extrapolation theorem \cite{RdF1984}). However, it is natural to ask whether the $L^1(w)\to L^{1,\infty}(w)$ estimate for $A_1$ weights $w$ has an extension for the critical $\la(p)$ and some $p>1$, and what the $p$-range of this
extension is. We give an affirmative answer to the first part of this question.

\begin{thm}\label{thm:A1}
Let $a>0$. For every $w\in A_1$ there exists an exponent $p_1(w)>1$ such that the operators $\cR_{a,t}^{\la(p)}$ are bounded from $L^p(w)$ to $L^{p,\infty}(w)$ for $1 \leq p < p_1(w)$, uniformly in $t>0$. Moreover,
$\lim_{t\to \infty} \| \cR^{\la(p)}_{a,t} f-f\|_{L^{p,\infty}(w) } =0$
for all $f\in L^p(w)$.
\end{thm}

The case $p=1$ in Theorem \ref{thm:A1} is Vargas' result \cite{Vargas96}; our contribution here corresponds to $p>1$.

In order to prove Theorem \ref{thm:A1}, we establish new sparse domination results for Bochner-Riesz means at the critical index, which will be presented in \S\ref{subsec:sparse}. These can be combined with a result of Frey and Nieraeth \cite{FreyNieraeth} to yield that, under the assumptions of the Bochner--Riesz conjecture in $d$ dimensions, the operators $\cR_{a,t}^{\la(p)}$ map $L^p(w)$ to $L^{p,\infty} (w)$ for $w\in A_1\cap \RH_\sigma$ and $p<1+\frac{d-1}{d+1}(1-\frac 1\sigma)$. This holds unconditionally if $d=2$ or if $d\geq 3$ and $\sigma$ belongs to a suitable range that includes $[1,\frac{d+3}{2}]$: see Section \ref{sec:weighted}.
Theorem \ref{thm:A1} will be a consequence of this, using the standard fact that every $A_1$ weight belongs to $\RH_\sigma$ for some $\sigma>1$.

It does not seem to be known whether $p < 1+ \frac{d-1}{d+1}(1-\frac{1}{\sigma})$ is the sharp $p$-range
in terms of the reverse H\"older exponent $\sigma$ in the $L^p(w) \to L^{p,\infty}(w)$ estimates. It would be interesting to investigate relevant examples.

\subsection{Sparse bounds} \label{subsec:sparse}
Let $\fD$ denote a {\it dyadic lattice} in the sense of the monograph by Lerner and Nazarov \cite[\S 2]{lerner-nazarov}. For a locally integrable function $f$, a cube $Q\in \fD$ and $1 \leq p < \infty$, let
$\jp{f}_{Q,p}=(|Q|^{-1}\int_Q|f(y)|^p\ud y)^{1/p}$.
Given $0<\ga<1$, the collection $\fS\in \fD$ is called {\em $\gamma$-sparse} if for every $Q\in \fS$ there is a measurable subset $E_Q\subset Q$ so that $|E_Q|\geq \ga|Q|$ and $\{E_Q: Q\in \fS\}$ is a collection of pairwise disjoint sets. Let $1\le p,q<\infty$. For a $\gamma$-sparse family $\fS$ of cubes we define a sparse form $\La_{p,q}^\fS$ and a corresponding maximal form $\La^*_{p,q}$ by
\begin{align} \label{eq:Lambda} \La_{p,q}^\fS (f_1,f_2) &= \sum_{Q\in \fS}|Q| \jp{f_1}_{Q,p}
\jp{f_2}_{Q,q},
\\ \label{eq:maxLambda}
\La^*_{p,q} (f_1,f_2) &=\sup_{\fS:\ga\text{-sparse}}\La^\fS_{p,q} (f_1,f_2),
\end{align}
where the sup is taken over all $\gamma$-sparse families (which are allowed to be subcollections of different dyadic lattices).
These definitions are of interest in the range $p \leq q<p'$.
A linear operator $T: C^\infty_c(\bbR^d) \to \cD'(\bbR^d)$ satisfies a $(p,q)$ sparse bound if for all $f_1,\,f_2\in C^\infty_c$ the inequality
\Be\label{eq:sparse-dom} |\inn{Tf_1}{f_2}|\le C \La^*_{p,q}(f_1,f_2) \Ee holds with some constant $C$ independent of $f_1$, $f_2$. In this case, we say that $T$ belongs to the space $\text{Sp}_\ga(p,q;\bbR^d)$ and we denote by $\|T\|_{\Sp_\gamma(p,q;\bbR^d)}$ the best constant in \eqref{eq:sparse-dom}.
The space $\Sp_\gamma(p,q;\R^d)$ does not depend on $\ga$ (\cf. \cite{lerner-nazarov}), so we usually keep $\gamma$ fixed and drop the subscript $\gamma$. If the dimension is clear from the context we will also drop the mention of $\bbR^d$.

Given $0 < \lambda \leq \frac{d-1}{2}$, let $\trapez_d(\la)$ denote the trapezoid with corners
\Be\label{eq:trapez}\begin{aligned} &P_1=(\tfrac{2\la+d+1}{2d},\tfrac{d-2\la-1}{2d}), &&P_2=(\tfrac{2\la+d+1}{2d}, \tfrac {d-1}{2d} + \tfrac{\la(d+1)}{d(d-1)}),
\\ &P_3=(\tfrac {d-1}{2d} + \tfrac{\la(d+1)}{d(d-1)}, \tfrac{2\la+d+1}{2d}),
&&P_4= (\tfrac{d-2\la-1}{2d}, \tfrac{2\la+d+1}{2d})\,.
\end{aligned} \Ee
One might conjecture that
sparse bounds for $\cR_{a,t}^\lambda$ and $\lambda >0$ hold for all $(\frac 1p,\frac 1q)\in \trapez_d(\la)$. This would be a strengthening of the Lebesgue mapping properties of $\cR_{a,t}^\la$; thus, one typically aims to only obtain the sparse improvement for values of $\la>0$ for which the Bochner-Riesz conjectured has been verified. It was observed in \cite{benea-bernicot-luque, lacey-mena-reguera} that for $(\frac 1p,\frac 1q)$ in the interior of the trapezoid, $(p,q)$-sparse bounds for $\cR_{a,t}^\lambda$ can be obtained via a single-scale analysis, with affirmative results depending on the partial knowledge on the Bochner-Riesz conjecture. Henceforth we will focus on the endpoint cases in which $(\frac{1}{p},\frac{1}{q})$ belongs to the boundary of $\trapez_d(\la)$. Furthermore, since sparse bounds are scale-invariant we will consider the case $t=1$, and write $\cR^\la_a=\cR^\la_{a,1}$.

The sharpness of the region $\trapez_d(\la)$ was first observed in \cite{benea-bernicot-luque}, and can also be deduced from general necessary conditions for sparse domination (\cf. \cite[Prop.1.9]{BRS-endpoints}).
The numerology of \eqref{eq:trapez} at $P_2$ is related to the
conjectured $L^p\to L^r$ bounds for Fourier multiplier operators with radial bumps on thin annuli (see \eqref{eq:LpLrhj} below), which have as necessary condition $\frac{1}{r}\geq \frac{d+1}{d-1}(1-\frac{1}{p})$ from Knapp examples. Note that $P_2=(\frac 1{p_{2}}, \frac 1{q_{2}})$ in \eqref{eq:trapez} satisfies
$1-\frac{1}{q_{2}}=\frac{d+1}{d-1}(1-\frac{1}{p_{2}})$
and that the vertical line segment $P_1P_2$ corresponds to the critical case where $\la=\la(p)$.

Almost sharp results at the critical line $P_1P_2$ were obtained in the case $\la=\frac{d-1}{2} $ (that is, $p=1$) by Conde-Alonso--Culiuc--Di Plinio--Ou \cite{conde-alonso-etal}; namely they proved a $(1,q)$ sparse bound for all $q>1$. Partial results on the line $P_1P_2$ were obtained in two dimensions by Kesler and Lacey \cite{KeslerLacey} whenever $0 < \lambda < 1/2$. At the critical $p_\la=\frac{4}{3+2\la}$, they showed a $\Sp(p_\la,q;\bbR^2)$ bound for $q>4$, thereby strengthening the weak type $(p_\la, p_\la)$ inequality in \cite{seeger-BRwt}. They posed as an open question
whether $\Sp(p_\la,q;\bbR^2)$ bounds hold in the range $\frac{4}{1+6\la}< q\le 4$.
Here we answer this question affirmatively, so that by duality we obtain a positive result
for the full interior of the sides $(P_1P_2)$ and $(P_3P_4)$ in $\trapez_2(\la)$. It remains open what happens on the top side $\overline {P_2P_3}$ of $\trapez_2(\la)$, except for the special case $\la_*=1/6$ covered in Theorem \ref{thm:STendpt} below.

\color{black}
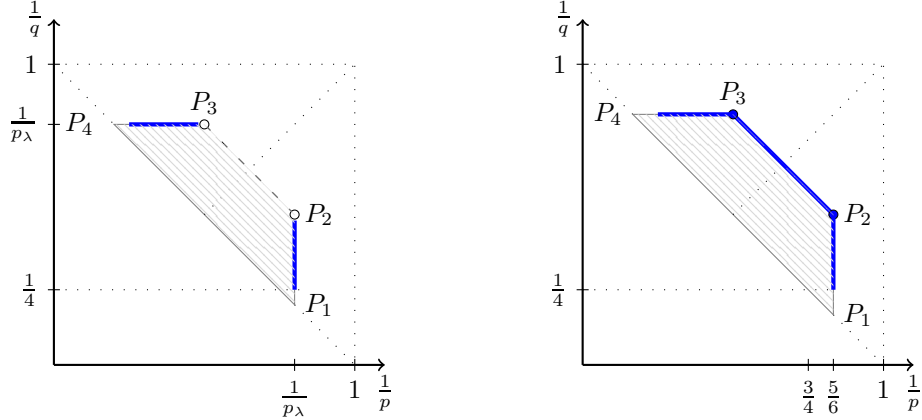
\begin{figure}

\begin{tikzpicture}[scale=2]

\begin{scope}[scale=2]

\draw[thick,->] (0,0) -- (1.1,0) node[below] {\small{$ \frac 1 p$}};
\draw[thick,->] (0,0) -- (0,1.15) node[left] {\small{$ \frac{1}{q}$}};

\draw[loosely dotted] (0,1) -- (1.,1.) -- (1.,0);
\draw[loosely dotted] (0,1) -- (1/5,4/5);
\draw[loosely dotted] (4/5,1/5) -- (1,0);
\draw[gray] (1/4, 4/5) -- (1/5,4/5) -- (0.5,0.5) -- (4/5,1/5) -- (4/5,1/4);

\draw (1,.02) -- (1,-.02) node[below] {\small{$1$}};
\draw (.02,1) -- (-.02,1) node[left] {\small{$1$}};

\draw (4/5,.02) -- (4/5, -.02) node[below] {\small{$\tfrac{1}{p_\lambda}$}} ;
\draw (.02,4/5) -- (-.02,4/5) node[left] {\small{$\tfrac{1}{p_\lambda}$}} ;

\draw[loosely dotted] (0.5,0.5) -- (1,1);

\draw[loosely dashed] (4/5,1/2) -- (1/2,4/5);
\draw[black,fill=white] (4/5,1/2) circle (.15mm);
\draw[black,fill=white] (1/2, 4/5) circle (.15mm);
\draw[ultra thick,color=blue] (4/5,1/4) -- (4/5,1/2-0.02);
\draw[ultra thick,color=blue] (1/2-0.02,4/5) -- (1/4,4/5);

\draw (.02,1/4) -- (-.02,1/4) node[left] {\small{$ \tfrac 14$}};
\draw[loosely dotted] (0,1/4) -- (5/6,1/4);

\draw (4/5,1/5) node[right] {\small{$P_1$}} ;
\draw (4/5,1/2) node[right] {\small{$P_2$}} ;
\draw (1/2,4/5) node[above] {\small{$P_3$}} ;
\draw (1/6,4/5) node[left] {\small{$P_4$}} ;

\fill[pattern=north west lines, pattern color= gray!30] (4/5,1/5) -- (4/5,1/2) -- (1/2, 4/5) -- (1/5,4/5) -- (4/5,1/5);

\begin{scope}[xshift=50]
\draw[thick,->] (0,0) -- (1.1,0) node[below] {\small{$ \frac 1 p$}};
\draw[thick,->] (0,0) -- (0,1.15) node[left] {\small{$ \frac{1}{q}$}};

\draw[loosely dotted] (0,1) -- (1.,1.) -- (1.,0);
\draw[loosely dotted] (0,1) -- (1/6,5/6);
\draw[loosely dotted] (5/6,1/6) -- (1,0);
\draw[gray] (1/6,5/6) -- (0.5,0.5) -- (5/6,1/6);

\draw (1,.02) -- (1,-.02) node[below] {\small{$1$}};
\draw (.02,1) -- (-.02,1) node[left] {\small{$ 1$}};

\draw (.02,1/4) -- (-.02,1/4) node[left] {\small{$ \tfrac 14$}};
\draw[loosely dotted] (0,1/4) -- (5/6,1/4);

\draw (3/4,.02) -- (3/4,-.02) node[below] {\small{$ \tfrac 34$}};

\draw (5/6,.02) -- (5/6, -.02) node[below] {\small{$\tfrac{5}{6}$}} ;

\draw[loosely dotted] (0.5,0.5) -- (1,1);

\draw[gray] (5/6,1/6) -- (5/6,1/2);
\draw[gray] (1/6,5/6) -- (1/2,5/6);
\draw[ultra thick, color=blue] (5/6,1/2) -- (1/2,5/6);
\draw[black,fill=blue] (5/6,1/2) circle (.15mm);
\draw[black,fill=blue] (1/2, 5/6) circle (.15mm);
\draw[ultra thick, color=blue] (5/6,1/4) -- (5/6,1/2);
\draw[ultra thick, color=blue] (1/2,5/6) -- (1/4,5/6);

\draw (5/6,1/6) node[right] {\small{$P_1$}} ;
\draw (5/6,1/2) node[right] {\small{$P_2$}} ;
\draw (1/2,5/6) node[above] {\small{$P_3$}} ;
\draw (1/6,5/6) node[left] {\small{$P_4$}} ;

\fill[pattern=north west lines, pattern color= gray!30] (5/6,1/6) -- (5/6,1/2) -- (1/2, 5/6) -- (1/6,5/6) -- (5/6,1/6);

\end{scope}

\end{scope}

\end{tikzpicture}

\caption{Sparse bounds for Riesz means $\cR^\lambda_{a,t}$ in $\R^2$ for any $0 < \lambda < 1/2$ on the left, and for the special case $\la=1/6$ on the right. The blue boundary segments correspond to the new content of Theorems \ref{thm:new-sparse-bound-2D} and \ref{thm:STendpt}, resp. Similar figures hold for $d \geq 3$ for a restricted range of $\lambda$; see Remarks after Theorem \ref{thm:blackboxvariant}, and Theorem \ref{thm:STendpt}.}
\label{fig:exponents}
\end{figure}

\begin{thm}\label{thm:new-sparse-bound-2D} Let $d=2$, $a>0$.
For $0<\la<1/2$, let $p_\la=\frac 4{3+2\la}$. Then we have
\[ \| \cR^\la_{a}\|_{\Sp (p_\la , q; \R^2)} <\infty, \quad \text{ for }
q>\tfrac{4}{1+6\la} .\]
\end{thm}

In higher dimensions, we obtain similar optimal results but only for a partial range of $\la$ away from $0$.
This is natural in view of the currently incomplete knowledge on
$L^p\to L^r$ bounds for Bochner-Riesz type operators. It will be convenient to formulate the sparse bounds conditional on off-diagonal Lebesgue space estimates for the Bochner-Riesz operators $\cR^\la_1$ (and unconditional for the Stein--Tomas exponent and some range beyond).

\begin{thm} \label{thm:blackboxvariant}
Let $d \geq 2$, $a>0$ and $\frac{2(d+1)}{d+3} \leq p_\circ < \frac{2d}{d+1} $.
Assume that for all $r_\circ\in [p_\circ, \frac{d-1}{d+1} p_\circ')$
the operator $\cR^\la_1$ maps $L^{p_\circ}(\bbR^d) \to L^{r_\circ} (\bbR^d) $ for all $\la>\la(r_\circ)$.
Then $\cR^{\la(p)}_a\in \Sp(p,q)$ for $1 \leq p<p_\circ$ and $q > q_{\mathit{opt}}:=\frac{(d-1)p}{d+1-2p}$.
\end{thm}

Several remarks are in order.

\begin{rem} \label{rem:listofremarks}

(i) The condition $q > q_{\mathit{opt}}$ is equivalent to saying that for the value $\lambda=\lambda(p)$, sparse bounds hold on the critical vertical line segment $P_1P_2$, except at the point $P_2$.

(ii) Theorem \ref{thm:new-sparse-bound-2D} is an immediate corollary of Theorem \ref{thm:blackboxvariant} due to the resolution of the Bochner--Riesz conjecture in 2 dimensions.

(iii) For $d \geq 3$, we are seeking to show that
\Be \label{eq:plaqla} \cR_a^\la\in \Sp(p_\lambda,q),\quad p_\la=\tfrac{2d}{d+1+2\la}, \quad q >\tfrac{2(d-1)d}{2\la(d+1)+(d-1)^2},\Ee
in a large range of $\lambda$. Since $\la(p_\la)=\la$ this corresponds to bounds on the critical endpoint segment $P_1P_2$, and the range of $q$ is optimal up to $P_2$. By Theorem \ref{thm:blackboxvariant} this can be achieved if we have a non-endpoint Bochner-Riesz $L^{p_\circ}\to L^{r_\circ}$ bound for some $p_\circ > p_\lambda$ and all $r_\circ \in [p_\circ, \frac{d-1}{d+1}p_\circ')$.
Instances for which this Bochner-Riesz hypothesis is known (and therefore our theorem is unconditional) are:
\begin{itemize}
\item The Stein--Tomas \cite{fefferman69} exponent $p_\circ=\frac{2(d+1)}{d+3}$. This leads to \eqref{eq:plaqla} for $\frac{d-1}{2(d+1)}<\la<\frac{d-1}{2}.$
\item The so-called bilinear Fourier restriction exponent, that is, for $p_\circ < \frac{2(d+2)}{d+4}$, proven in \cite{ChoKimLeeShim2005} for $\rho(\xi)=|\xi|$. This leads to \eqref{eq:plaqla} for $\frac{d-2}{2(d+2)}<\la<\frac{d-1}{2}$.
\item The exponents obtained through multilinear restriction:
$p_\circ < \frac{2(d^2+3d-2)}{d^2+5d-2}$ for even $d \geq 4$, and $p_\circ < \frac{2(d^2+4d-1)}{d^2+6d+1}$ for odd $d \geq 5$,
proven in \cite{Kwon-Lee} via the oscillatory integral estimates in \cite{GHI}; these exponents correspond to the dual exponents to $q_\circ$ in \cite[(1.15)]{Kwon-Lee}. This extends
\eqref{eq:plaqla} to a range of $\la$'s smaller than $\frac{d-2}{2(d+2)}.$
\end{itemize}

(iv) Key to Theorem \ref{thm:blackboxvariant} is Theorem \ref{thm:blackbox-sparse}, which replaces the non-endpoint Bochner--Riesz boundedness assumption by an endpoint variant for certain vector-valued functions, labelled $\VV(p,r)$ in Definition \ref{def:VV}. For further details see \S\ref{sec:hypotheses}.

(v) Theorem \ref{thm:blackboxvariant} follows from a more general result that only imposes the $L^{p_\circ} \to L^{r_\circ} $ non-endpoint inequalities for the Bochner-Riesz operator in Theorem \ref{thm:blackboxvariant} for a specific $r_\circ$ (instead of the almost optimal range of $r_\circ$). Such a theorem is formulated as Theorem \ref{thm:blackboxsparse-r0} below.

\end{rem}

In Theorems \ref{thm:new-sparse-bound-2D} and \ref{thm:blackboxvariant} it remains open whether the $\Sp(p_\la,q_{opt,\la})$ bound holds with
$q_{\mathit{opt},\la}:=\tfrac{2(d-1)d}{2\la(d+1)+(d-1)^2}$, that is, at the endpoint $P_2$.
We can prove this when the Bochner-Riesz index is equal to $\la_*=\frac{d-1}{2(d+1)} $; in this instance $q_{\mathit{opt},\la}=2$. This corresponds to the endpoint in the Stein-Tomas restriction theorem and gives us added flexibility to use $L^2$ methods. We also obtain the corresponding sparse bounds on the full top side $\overline{P_2P_3}$, thereby proving the optimal sparse bounds in the closed trapezoid $\trapez_d(\lambda_*)$, for this special case.
\begin{thm} \label{thm:STendpt}
Let $d\ge 2$, $a>0$. Let $\la_* =\frac{d-1}{2(d+1)}$ and $(\frac 1p,\frac 1q)\in \trapez_d(\la_*)$. Then $\cR^{\la_*}_a\in \Sp(p,q;\bbR^d)$.

\end{thm}

The main novelty of this paper is the introduction of a refined decomposition of the Riesz means $\cR_{a,t}^\lambda$ which has improved kernel localization properties in the spirit of Christ \cite{Christ-rough} but still retains good Fourier support properties. This allows to combine the two existing sparse endpoint approaches for $\cR_{a,t}^\lambda$, that is, the $p=1$ result of \cite{conde-alonso-etal}, and the partial two-dimensional result for $p>1$ of \cite{KeslerLacey}.
When $q=2$ one can further exploit the Fourier orthogonality properties of the decomposition to obtain Theorem \ref{thm:STendpt}.

\subsection*{Notation} \label{sec:notation} {We list some frequently used notation.}
\begin{itemize}
\item[$\circ$] {\it Families of dyadic cubes.} We let $\fD$ be a dyadic lattice of cubes in the sense of Lerner and Nazarov \cite{lerner-nazarov}.
We use $ \fQ$ for general subcollections of $\fD$.
We use the notation $\fW$ if such a subcollection is obtained by a Whitney decomposition of an open set with certain quantitative properties. We use $\fS$ for sparse families of dyadic cubes.
The sidelength of a dyadic cube $Q$ is denoted by $2^{L(Q)}$ with $L(Q)\in \bbZ$.
For a collection $\fQ$ of cubes we denote by $\fQ_j$ the collection of cubes in $\fQ$ with sidelength $2^j$. Similarly $\fQ_{\ge j}$ denotes the cubes $Q\in\fQ$ with $L(Q)\ge 2^j$. Analogously we define $\fQ_{\le j}, \fQ_{>j}, \fQ_{<j}$.
\item[$\circ$] {\it Normalized bump functions.} For $M\ge 1$ let $\cY_M$ be the class of all $C^{M}$ functions $\chi$ supported on
$(\frac 12,2)$ such that $\|\chi\|_{C^{M}}:=\sum_{\nu=0}^M\| \chi^{(\nu)} \|_\infty \le 1$.

\item[$\circ$] {\it Riesz multipliers.} We write $h_\la(\varrho) =\chi(\varrho)(1-\varrho)_+^\la$
with $\chi\in C^\infty_c((1/2,2)) $ and $\chi(\varrho)=1$ near $\varrho=1$ (see \eqref{initial-dec} below). The decomposition $h_\la=\sum_{\ell=0}^\infty h_{\la,\ell}$ is defined in \eqref{eq:hlaells}.

\end{itemize}

\subsection*{Outline of the paper} In \S\ref{sec:hypotheses}
we formulate refined versions of Theorem \ref{thm:blackboxvariant} involving Bochner--Riesz type inequalities for certain vector-valued functions and discuss how Theorems \ref{thm:new-sparse-bound-2D} and \ref{thm:blackboxvariant} follow from them.
In \S\ref{sec:decomp} we introduce a crucial decomposition of the Riesz multipliers.
In \S\ref{sec:Mainestimates} we shall state the main technical estimates
used in the sparse domination argument, with a key result (Theorem \ref{thm:Xi}) proved in
\S\ref{sec:ProofofMainProp}. Theorem \ref{thm:blackbox-sparse}, which is the main black-box sparse domination result, is proved in \S\ref{sec:sparse-partI}. The endpoint sparse domination results for the Riesz means at the index $\la=\frac{d-1}{2d+2}$ (Theorem \ref{thm:STendpt}) are treated in \S\ref{sec:auxiliary-for Thm1.3} and \S\ref{sec:sparse-partII}. Some consequences for weak type inequalities with weights, including the proof of Theorem \ref{thm:A1}, are discussed in \S\ref{sec:weighted}.
\subsection*{Acknowledgements} This research was supported through the program {Oberwolfach Research Fellows} by Mathematisches Forschungsinstitut Oberwolfach in 2023. The authors were supported in part by National Science Foundation grants
DMS-1954479 (D.B.), DMS-2154835 (J.R.), DMS-2054220 (A.S.),
and by the AEI grants RYC2020-029151-I and PID2022-140977NA-I00 (D.B.).
\section{A refined version of Theorem \ref{thm:blackboxvariant} and Bochner-Riesz type bounds for vector-valued functions}\label{sec:hypotheses}
We next formulate a more refined version of Theorem \ref{thm:blackboxvariant} which only involves a Bochner-Riesz non-endpoint $L^{p_\circ}\to L^{r_\circ}$ assumption for a specific value of $r_\circ$, as opposed to all values of $r_\circ \in [p_\circ, \frac{d-1}{d+1} p_\circ')$.

Let $d \geq 2$, $\frac{2(d+1)}{d+3}\le p_\circ<\frac{2d}{d+1}$, and $p_\circ\leq r_\circ \leq \frac{d-1}{d+1} p_\circ'.$
Define the exponent $r_*(p,p_\circ, r_\circ)$ and its dual $q_*(p,p_\circ, r_\circ) $ by
\Be \label{eq:r*def} 1-\frac{1}{q_*(p,p_\circ,r_\circ)}=\frac{1}{r_*(p,p_\circ,r_\circ)} :=
\begin{cases}
\frac{ \tfrac{1}{r_0} (\tfrac{d+3}{2(d+1)} -\tfrac 1p) +\frac 12 (\tfrac 1p-\tfrac 1{p_\circ}) } {\tfrac{d+3}{2(d+1)}-\frac 1{p_\circ} } &\text{ if $\frac{2(d+1)}{d+3} \le p<p_\circ$},
\\
\frac{d+1}{d-1}(1-\frac 1{p}) &\text{ if $1\le p \le \frac{2(d+1)}{d+3}$}.
\end{cases}
\Ee
These are motivated by interpolation numerology between the pairs $(\frac{1}{p_\circ}, \frac{1}{r_\circ})$ and $(\frac{d+3}{2(d+1)}, \frac{1}{2})$ when $\frac{2(d+1)}{d+3} < p < p_\circ$.
Moreover, $q_*(p,p_\circ,r_\circ) = \frac{(d-1)p} {d+1-2p} $ when $p\le \frac{2(d+1)}{d+3}$ and
\Be\label{eq:q*lim}
\lim_{r_\circ\to \tfrac{d-1}{d+1} p_\circ'}
r_*(p,p_\circ, r_\circ) = \tfrac{d-1}{d+1} p' \, ,
\qquad
\lim_{r_\circ\to \tfrac{d-1}{d+1} p_\circ'}
q_*(p,p_\circ, r_\circ) = \tfrac{(d-1)p}{d+1-2p} \,
\Ee
for all $p<p_\circ$.
The refinement of Theorem
\ref{thm:blackboxvariant} is as follows.

\smallskip
\begin{thm} \label{thm:blackboxsparse-r0}
Let $d\ge 2$, $a>0$, $\frac{2(d+1)}{d+3} \le p_\circ < \frac{2d}{d+1} $ and $p_\circ \leq r_\circ \leq \frac{d-1}{d+1} p_\circ'$. Assume that
the operator $\cR^\la_1$
maps $L^{p_\circ}(\R^d) $ to $L^{r_\circ}(\R^d)$ for all $\la>\la(r_\circ)$.
Then $\cR_a^{\la(p)} \in \Sp(p,q)$ for all $1 \leq p<p_\circ$ and $q>q_*(p,p_\circ,r_\circ).$
\end{thm}

Note that Theorem \ref{thm:blackboxvariant}
follows from Theorem \ref{thm:blackboxsparse-r0} by using the second limiting relation in \eqref{eq:q*lim}.

\subsection{Auxiliary inequalities on vector-valued functions}
The proof of Theorem \ref{thm:blackboxsparse-r0} relies on certain inequalities for families of operators of Bochner-Riesz type acting on certain vector-valued $L^p$ spaces, depending on admissible parameters $p$ and $r$. We give a formal statement in the following definition;
the set of normalized bump functions $\cY_M$ is defined in the notation section above.

\begin{definition}\label{def:VV} Let $1\le p\le r< \infty$. Let $\VV(p,r)$ denote the following statement. There is $M>0$ such that for all collections $\chi_j$ of functions in $\cY_M$
the inequality
\Be \label{prvv-conclusion} \Big\| \sum_{j > 0} 2^{j\frac{d+1}{2}} \chi_j(2^j(1-\rho(D))) \big[\sum_{Q\in \fD_j} f_Q\big] \Big\|_{L^r (\bbR^d)}\le C_{p,r,d}
\Big( \sum_{Q\in \fD} |Q| \|f_Q\|_{L^p(\bbR^d)}^r \Big)^{1/r}
\Ee
holds for all families $\{f_Q\}_{Q\in \fD}$ of $L^p$ functions $f_Q$ with $\supp(f_Q)\subset \overline{Q}$.
\end{definition}

Applying \eqref{prvv-conclusion} to a family of cubes of a fixed sidelength $2^j$ shows that $\VV(p,r)$ yields the multiplier bound for a single bump $\chi \in \cY_M$
\Be\label{eq:LpLrhj}\| \chi(2^j(1-\rho(D))) \|_{L^{p} \to L^{r}} =O(2^{j\la(r) })\Ee
which is conjectured to hold for
$1 \le p< \frac{2d}{d+1}$, $p\leq r \le \frac{d-1}{d+1}p'$; recall $\la(r)=d(\frac 1r-\frac 12)-\frac 12$.
The inequalities \eqref{prvv-conclusion} are a multi-scale version of \eqref{eq:LpLrhj}.

The main technical result that is used to prove essentially sharp sparse bounds for $\cR^{\la(p)}_a$ reduces the conclusion of sparse bounds to estimates of $\VV$-type.

\begin{thm} \label{thm:blackbox-sparse}
Let $d\ge 2$, $a>0$, $\frac{2(d+1)}{d+3}\le p_\circ<\frac{2d}{d+1}$ and $p_\circ \leq r_\circ \leq \frac{d-1}{d+1}p_\circ'$.
Assume that $\VV(p, r)$ holds for all $p\in[\frac{2(d+1)}{d+3}, p_\circ)$ and $r\in [p, r_*(p,p_\circ,r_\circ))$. Then
$\cR^{\la(p)}_{a} \in \Sp(p,q)$ for $1\le p <p_\circ$, $q > q_*(p,p_\circ, r_\circ)$.
\end{thm}

The proof of Theorem \ref{thm:blackbox-sparse} will be given in \S\S\ref{sec:Mainestimates}--\ref{sec:sparse-partI}. The conclusion of Theorem \ref{thm:blackbox-sparse} also holds with $p=\frac{2(d+1)}{d+3} $ and $q=2$; this is the statement of Theorem \ref{thm:STendpt}, which is proved in \S\S\ref{sec:auxiliary-for Thm1.3}--\ref{sec:sparse-partII}.

\subsection{\texorpdfstring{Instances in which $\VV(p,r)$ holds and relation with Theorem \ref{thm:blackboxsparse-r0}}{Instances in which the hypothesis holds}}\label{sec:known}
In order to fill Theorem \ref{thm:blackbox-sparse} with content we first gather known results regarding $\VV_{d}(p,r)$. The following results are available in the literature.

\begin{enumerate}[(i)]
\item
For $d\ge 2$, $\VV(p,r)$ holds for $1\le p\le \frac{2(d+1)}{d+3}$, and $p\le r\le 2.$

\item
For $d=2$, $\VV(p,r) $ holds for $1\le p<4/3$, $p\le r<\min\{p'/3,2\}$.

\item
Suppose that $1<p_\circ<\frac{2d}{d+1}$ and suppose that $\cR^\la_1$ is bounded on $L^{p_\circ}$ for all $\la>\la(p_\circ)$.
Then $\VV(p,p)$ holds for $1\le p<p_\circ$.
\end{enumerate}

Part (i) of this statement for $r=2$ is just Lemma \ref{lem:ST-style} and it is a standard consequence of the $L^2$-restriction theorem. The statement for $p\le r<2$ is in \cite{Seeger-Indiana}, in the slightly more general setup for spectral multipliers on compact manifolds. Part (ii) for $p=r$ is an immediate consequence of a vector-valued inequality in \cite{seeger-BRwt}, the general case follows by interpolating with the result in part (i). The conditional result in part (iii) was proved by Tao in his paper \cite{Tao-Indiana1998} on weak type $(p,p)$ estimates for Bochner-Riesz means.

The bounds (i)-(iii) can be combined with Theorem \ref{thm:blackbox-sparse} to deduce endpoint sparse bounds for $\cR^\lambda_a$.

\begin{enumerate}[(i')]
\item The $\VV$ inequalities in two dimensions stated in (ii) yield Theorem \ref{thm:new-sparse-bound-2D} (without passing through Theorem \ref{thm:blackboxvariant}).
\item The $\VV$ inequalities in the Stein-Tomas-range in (i) yield \eqref{eq:plaqla} for $\frac{d-1}{2(d+1)}<\la<\frac{d-1}{2}$ (that is, the conclusion of Theorem \ref{thm:blackboxvariant} if one inputs $p_\circ=\frac{2(d+1)}{d+3}$).
\item The $\VV(p,p)$ bounds by Tao in (iii) for $\frac{2(d+1)}{d+3} < p < p_\circ$ yield some endpoint $(p,q)$-sparse bounds on a portion of the segment $P_1P_2$. However, this does not yet lead to close to optimal bounds for $q$ in the sparse bounds. This phenomenon also occurs in the work by Kesler and Lacey \cite{KeslerLacey} in two dimensions who essentially work with a $\VV(p,p)$ input bound from \cite{seeger-BRwt}.
\end{enumerate}

In order to effectively prove sparse bounds in the whole (open) segment $P_1P_2$ beyond the Stein-Tomas range one needs to obtain an off-diagonal version of Tao's theorem. Tao \cite[p. 1111]{Tao-Indiana1998} raises this question on whether there are such $L^p\to L^r$ versions of his theorem.
Away from the critical line $r=\frac{d-1}{d+1}p'$ such versions can be obtained by using modifications of his proof which relies on $\varepsilon$-removal arguments. The interested reader can find the details in \cite{BRS-expository}.

\begin{thm}[{\cite[Theorem 1.2]{BRS-expository}}]\label{thm:Taovariant}
Let $d \geq 2$, $\frac{2(d+1)}{d+3} < p_\circ < \frac{2d}{d+1} $ and $p_\circ \leq r_\circ \leq \frac{d-1}{d+1}p_\circ'$. Assume that the operator $\cR^\la_1$
maps $L^{p_\circ} \to L^{r_\circ} $ for all $\la>\la(r_\circ).$
Then $\VV (p,r)$ holds for $\frac{2(d+1)}{d+3}\le p < p_\circ$, $p \le r<r_*(p,p_\circ, r_\circ)$.
\end{thm}

It is clear that Theorem \ref{thm:blackboxsparse-r0} is now a consequence of Theorems \ref{thm:blackbox-sparse} and \ref{thm:Taovariant}, which in turn implies Theorem \ref{thm:blackboxvariant}.

\section{Decompositions of Riesz means} \label{sec:decomp}

We introduce a decomposition of the Riesz multipliers which has strong localization properties on both the kernel and the multiplier side and will play a crucial role in the estimates needed to establish the sparse domination results.
We remark that rudimentary versions of this decomposition already featuring variants of condition \eqref{eq:la-canc} below go back to
\cite{Christ-BR87} and \cite{tao-STendpoint}. However, these have weaker conclusions that we found to be insufficient for our arguments in the proof of Theorem \ref{thm:blackbox-sparse}.

We start with some basic reductions. Let $\widetilde{\chi}\in C^\infty$ be supported in $(1/2,2)$ such that $\widetilde{\chi}(\varrho)=1$ in a neighborhood of $1$.
We note that for all $1 \leq p < \infty$, a standard sparse $\Sp(p,p)$ bound holds for the Fourier multiplier operator with multiplier
$ (1-\widetilde{\chi}(\rho(\xi))(1-\rho(\xi)^a )_+^\la$.
Indeed, note that for $\alpha \in \N_0^d$ with $|\alpha|\ge 1$ we have
\[ \big|\partial_\xi^{\alpha} \big[ (1-\widetilde{\chi}(\rho(\xi))(1-\rho(\xi)^a)_+^\la \big]\big|\lc_\alpha ( 1+|\xi|^{a-|\alpha|} +|\xi|^{1-|\alpha|} )\]
which together with the support property implies a kernel estimate $ O((1+|x|)^{-d-\eps} )$ with $\eps<
\min\{1,a\}$ for the underlying kernel. We therefore focus on the essential contribution, corresponding to the multiplier
$\widetilde{\chi}(\rho(\xi))(1-\rho(\xi) ^a)_+^\la $. We also note that we can assume without loss of generality that $a=1$. This is because
\begin{subequations}\label{initial-dec}
\Be\chi_{a,\la}(\varrho) =\widetilde{\chi}(\varrho ) \frac{ (1-\varrho^a)^\la }{ (1-\varrho)^\la} \Ee is smooth near $\varrho=1$ and thus it suffices to just consider the multiplier $h_\la(\rho(\xi))$ with
\Be\label{mla=hla} h_\la(\varrho) =\chi_{a,\la}(\varrho) (1-\varrho )_+^\la,\Ee
\end{subequations}
where for fixed $a$, the family $\{\chi_{a,\la}: |\la|\le d\} $ is a bounded collection of $C^\infty_c$ functions supported in $(\frac 12,2)$. We shall write $\chi\equiv \chi_{a,\la}$ in what follows.

The following lemmas will be useful in further splitting the multiplier $h_\lambda$.

\begin{lemma} \label{lem:cancellations}
Let $\la>0$, $N_\circ\in \bbN$. There exists an even $C^\infty_c(\bbR)$ function $\Phi_\circ$ such that
$\Phi_\circ(s)=1$ for $|s|\le 1/2$ and $\Phi_\circ(s) =0$ for $|s|\ge 1$ and, in addition, \Be \label{eq:la-canc} \int_0^\infty \varrho^\la\big(\tfrac{ d}{d\varrho} \big)^j \widehat {\Phi_\circ} (\varrho) \ud\varrho=0 \quad \text{for }\,\, j=0,1,\dots, N_\circ, \,\, j\neq \la.\Ee
\end{lemma}

\begin{proof}

We consider the interval $I=[-7/4,-5/4]$ and $L^2(I)$ with the usual scalar product. Let $\bbV$ be the span of the functions $s\mapsto |s|^{-\la+j} \bbone_{[-7/4,-5/4]}$ where $j=0,\dots, N_\circ$ with $j\neq \la$. We pick $u\in L^2$ supported on $I$ such that
\[\int_I u(s) \ud s=1\] and such that $u\in \bbV^\perp $; that is, we have
$\int_I u(s) |s|^{j-\la}\ud s=0$ for integers $0\le j\le N_\circ$ with $j\neq \la$. Note that also $\int_{-\infty}^0 u(s/t) |s|^{j-\la}\ud s=0$ for those $j$ and all $t>0$. This suggests that in order to regularize $u$ we should work with a multiplicative convolution.
Let
$0< \varepsilon < 1/8$ and $w\in C^\infty_c$ supported in $(1-\eps,1+\eps)$ with $\int w(x)\ud x=1$. Define for $x<0$
\[ U(x)= \int_{0}^\infty u\Big(\frac xt\Big) w(t)\frac{\ud t}{t}= \int_0^\infty u(-t) w\Big(\frac {-x}t\Big) \frac {\ud t} t\] and set, for $x>0$, $U(x)=-U(-x)$, and $U(0)=0$. In view of the support properties of $u$ and $w$, we see that $U$ is an odd $C^\infty_c $ function supported in $(-2,-1)\cup(1,2)$ and we have
\begin{equation}
\int_{-\infty}^0 U(s) |s|^{j-\la} \ud s= \int_{1-\varepsilon}^{1+\varepsilon} w(t) t^{j-\la} \ud t \int_I u(s)|s|^{j-\la} \ud s=0
\end{equation}
for all $j \in \{0,1, \dots, N_\circ\}\backslash \{\lambda\}$, since $I \subseteq [-2/t,-1/t]$ for $t \in (1-\varepsilon, 1+\varepsilon)$. Similarly,
for $-1\le x\le 1$,
\begin{align*} &\int_{-\infty} ^x U(s) \ud s= \int_{-\infty} ^{-1} U(s) \ud s
=\int_{1-\eps}^{1+\eps} w(t) \ud t \int_I u(s) \ud s =1.
\end{align*}
We now define
\[\Phi_\circ (x) :=\int_{-\infty}^x U(s) \ud s.\] From the above calculations, we obtain that $\Phi_\circ $ is an even $C^\infty_c$ function supported in $(-2,2)$ such that $\Phi_\circ(x)=1$ for $|x|\le 1$ and
\Be \label{cancderivPhi} \int_1^2 \Phi_\circ'(x) x^{j-\la} \ud x=0,\quad j\in \{0,1,\dots, N_\circ\}\setminus \{\la\}.\Ee
We will next show that \eqref{cancderivPhi} implies \eqref{eq:la-canc}.

Recall that for $\la>-1$ the distributional Fourier transform of $\varrho_+^\lambda/\Gamma(\la+1)$ is the distribution $e^{-i\pi (\la+1)/2 }(\xi-i0) ^{-\la-1} $; see for example \cite[p.167]{hormander-ALPDO-1}. This means that for Schwartz functions $\phi$ we have
\Be\label{FTrhola} \int_0^\infty \widehat \phi(\varrho) \frac{\varrho^\la}{\Gamma(\la+1)} \ud \varrho= e^{-i\pi (\la+1)/2 }\lim_{y\to 0+} \int_{-\infty}^\infty \phi (x) (x-iy)^{-\la-1} \ud x
\Ee
and the limit exists (\cf. \cite[Thm 3.1.11]{hormander-ALPDO-1});
moreover the tempered distribution $(x-i0)^{-\la-1} $ is identified with the function $x^{-\la-1} $
in $(0,\infty)$.
The previous display gives
\begin{equation*}
\int_0^\infty \varrho^\la \big (\frac d{d\varrho} \big)^j \widehat{\Phi_\circ}(\varrho) \ud \varrho= \frac{\Gamma(\la+1)}{e^{i\pi (\la+1)/2 }}\lim_{y\to 0+} \int_{-\infty}^\infty (-ix)^j \Phi_\circ (x) (x-iy)^{-\la-1} \ud x.
\end{equation*}
In view of the existence of the boundary value distribution $(x-i0)^{-\la-1}$ it is immediate that for $j\ge 1$
\[\lim_{y\to 0+} \int_{-\infty}^\infty (-i)^j \big( x^j - (x-iy)^j\big) \Phi_\circ (x) (x-iy)^{-\la-1} \ud x=0; \] indeed the integral can be written as $\sum_{k=1}^j y^k \int \phi_k (x) (x-iy)^{-\la-1} \ud x$ with suitable test functions $\phi_k$. Therefore we get
\begin{equation*}
\int_0^\infty \varrho^\la \big (\frac d{d\varrho} \big)^j \widehat{\Phi_\circ}(\varrho) \ud \varrho = \frac{ (-i)^j \Gamma(\la+1)}{ e^{i\pi (\la+1)/2 }} \lim_{y\to 0+} \int_{-\infty}^\infty \Phi_\circ (x) (x-iy)^{j-\la-1} \ud x.
\end{equation*}
Integrating by parts and using $j\neq \la$ we also get for fixed $y>0$
\begin{multline*} \int_{-\infty}^\infty \Phi_\circ (x) (x-iy)^{j-\la-1} \ud x = -\frac{1}{j-\la} \int_{-\infty}^\infty\Phi_\circ'(x) (x-iy)^{j-\la} \ud x
\\
= -\frac{1}{j-\la} \Big( \int_1^2 \Phi_\circ' (x) (x-iy)^{j-\la} \ud x+ \int_{-2}^{-1} \Phi_\circ'(x) (x-iy)^{j-\la} \ud x\Big).
\end{multline*}
For $j\neq \la$ the boundary value distribution $(x-i0)^{j-\la} $ is identified with the function $x^{j-\la}$ on $(0,\infty) $ and with the function $(e^{-i\pi}|x|)^{j-\la} $ on $(-\infty,0)$. Also recall that $\Phi_\circ'\equiv U$ is odd. Combining the above observations we obtain after taking the limit,
\begin{equation*}
\int_0^\infty \varrho^\la \big (\frac d{d\varrho} \big)^j \widehat \Phi(\varrho) \ud \varrho= \frac{\Gamma(\la+1)}{ e^{i\pi (\la+1)/2 }} \frac{(-1)^{j+1} }{j-\la} (1- e^{i\pi(j-\la)})
\int_1^2 \Phi_\circ'(x) x ^{j-\la} \ud x
\end{equation*}
and \eqref{eq:la-canc} follows from \eqref{cancderivPhi}.
\end{proof}

The condition \eqref{eq:la-canc} fails when $j=\la$. In this case we have instead
\begin{lemma} \label{lem:elem-canc}
For all even Schwartz functions $\phi$, and $j=0, 1,2,3,\dots$
\Be \label{eq:la-canc-pi} \int_0^\infty \varrho^j\Big(\frac{ d}{d\varrho} \Big)^j \widehat {\phi} (\varrho) \ud\varrho=(-1)^j\pi j!\phi(0).
\Ee
\end{lemma}
\begin{proof} A $j$-fold integration by parts yields that
\[ \int_0^\infty \varrho^j\Big(\frac{ d}{d\varrho} \Big)^j \widehat {\phi} (\varrho) \ud\varrho=(-1)^jj!\int_0^\infty \widehat \phi (\varrho)\ud\varrho= \frac {(-1)^j} 2 j! \int_{-\infty}^\infty \widehat \phi (\varrho)\ud\varrho,\]
where the second identity follows since $\widehat \phi$ is even. The claim now follows from the Fourier inversion formula.
\end{proof}

As an immediate consequence of Lemma \ref{lem:cancellations}, and in the case of integer $\la$ also Lemma \ref{lem:elem-canc}, we obtain
\begin{cor} \label{cor:cancellations}
Let $\la>0$, $N_\circ\in \bbN$. Let $\Phi_\circ$ be as in Lemma \ref{lem:cancellations} and let \Be\label{eq:Psidef} \Psi(x)=\Phi_\circ (x/2)-\Phi_\circ(x) .\Ee Then $\Psi$ is an even $C^\infty_c(\bbR)$ function such that $\Psi(s)=0$ for $|s|\le 1/2$ and $\Psi(s) =0$ for $|s|\ge 2$ and such that \Be \label{eq:la-canc-Psi} \int_0^\infty \varrho^\la\big(\tfrac{ d}{d\varrho} \big)^j \widehat {\Psi} (\varrho) \ud\varrho=0, \quad j=0,1,\dots, N_\circ.\Ee
\end{cor}

We now decompose
$\cF_{\bbR}^{-1}[(1-\varrho)_+^\la]$ dyadically, using the functions $\Phi_\circ$, and dilates of $\Psi$ as in \eqref{eq:Psidef}.
In the following definition (and then throughout the paper) we will assume that $N_\circ$ in \eqref{eq:la-canc} satisfies $N_\circ>d$. We get
\begin{subequations}\label{eq:hlaells}
\begin{align}\label{hla-decomposition} h_\la&=\sum_{\ell=0}^\infty h_{\la,\ell} \quad \text { with }
\\
\label{eq:hlaell-0}
h_{\la,0}(\varrho) &= \frac {\chi(\varrho)}{2\pi}\int_{-\infty}^\infty (1-u)_+^\la \widehat{\Phi_\circ}(\varrho-u) \ud u,
\\
\label{eq:hlaell}
h_{\la,\ell}(\varrho) &= \frac {\chi(\varrho)} {2\pi}\int_{-\infty}^\infty (1-u)_+^\la 2^{\ell-1}\widehat{\Psi}(2^{\ell-1} (\varrho-u) )
\ud u, \quad \ell>0.
\end{align}
\end{subequations}

\begin{lemma} \label{lem:ptwisebdhellla}
For all $N_1\in \bbN$ and
for all $\alpha \in \N_0^d$ with $|\alpha|\le N_1$
\Be \label{eq:hlaellnear1}
|\partial_\xi^\alpha [h_{\la,\ell} \circ \rho](\xi) | \le C_{N_1,\alpha} 2^{-\ell (\lambda-|\alpha|)} (1+2^\ell |1-\rho(\xi)|)^{-N_1}.
\Ee
Let $\ell\ge 0$, $N_\circ$ as in Corollary \ref{cor:cancellations} and $|\alpha|\le N_\circ$. Then
\Be\label{eq:hlaellcancat1-deriv} |\partial_\xi ^\alpha [h_{\la,\ell} \circ\rho] (\xi) | \lc 2^{-\ell (\la-|\alpha|) } |2^\ell(1-\rho(\xi) ) |^{N_\circ+1-|\alpha|} \text{ if $|1-\rho(\xi) |\le 2^{-\ell}$}. \Ee
\end{lemma}
\begin{proof} Repeated integration by parts yields \eqref
{eq:hlaellnear1}.
We use Corollary \ref{cor:cancellations} and Taylor's theorem to compute for $\ell\ge 1$
\begin{align*} &h_{\la,\ell} (\varrho) = \chi(\varrho)\frac{1}{2\pi}\int_{-\infty}^1 (1-u)^\la 2^{\ell}\widehat{\Psi}(2^\ell(\varrho-u) )\ud u
\\
&= \frac{\chi(\varrho)}{2\pi}\int_{-\infty}^1 (1-u)^\la 2^{\ell}
\big [ \widehat{\Psi}(2^\ell(\varrho-u) )
-\sum_{j=0}^{N_\circ} \frac{ (2^\ell(\varrho-1))^j}{j!} \widehat\Psi^{(j)} (2^\ell(1-u))\big]
\ud u
\\
&=(2^\ell (\varrho-1))^{N_\circ+1} \frac{\chi(\varrho) }{2\pi} \,\,\times\\ &\qquad\qquad\int_0^1 \frac{(1-\sigma)^{N_\circ}}{N_\circ!} \int_{-\infty}^1 (1-u)^\la 2^\ell \widehat\Psi^{(N_\circ+1)}(2^\ell(1-u+\sigma (\varrho-1))) \ud u \ud \sigma
\end{align*}
which in turn gives $
|h_{\la,\ell} (\varrho) | \lc 2^{-\ell \la} |2^\ell(1-\varrho) |^{N_\circ+1} $ for $|1-\varrho|\le 2^{-\ell}$. Thus, setting $\varrho=\rho(\xi)$, \eqref{eq:hlaellcancat1-deriv} follows for $\alpha=0$.
A similar calculation follows for $\partial_\varrho^j h_{\la,\ell}$ and then \eqref{eq:hlaellcancat1-deriv} for higher derivatives follows by applications of the multivariate Leibniz rule and the Fa\`a di Bruno's formula.
\end{proof}

Since $\nabla\rho$ is homogeneous of order zero and since $\nabla\rho$ does not vanish on $\partial\Omega$ there are two positive constants $c_0$, $C_0$ such that $C_0\ge 1$ and
\begin{subequations}
\Be\label{eq:nablarho-bds} c_0<|\nabla\rho(\xi)|\le C_0 \text{ for all $\xi\neq 0$.}\Ee Later in the paper it will also be useful to fix a positive integer $\nc$ such that \Be \label{eq:ncircdef} 2^{d+4} C_0< 2^{\nc} .\Ee
\end{subequations}
We next study the properties of the convolution kernels
\begin{equation*}
K_{\lambda,\ell}(x):=\mathcal{F}^{-1}[h_{\lambda,\ell} \circ \rho](x), \qquad \ell \geq 0.
\end{equation*}
The next lemma shows that for $\ell > 0$, the kernels $K_{\lambda, \ell}$ are essentially supported in $\{x: c_0 2^{\ell-2} \le |x|\le C_0 2^{\ell+2} \}$. No curvature assumption is necessary here.
\begin{lemma} \label{lem:error-kernelest} For all $N\in \bbN$, \[|K_{\la,\ell}(x) |\lc_N\begin{cases} |x|^{-N} &\text{ for $|x|\ge 2^{\ell+2} C_0$} \\
2^{-\ell N } &\text{ for $|x|\le 2^{\ell-2} c_0$.}
\end{cases} \]
\end{lemma}

\begin{proof} The statement for $\ell=0$ follows from integration by parts. We thus assume $\ell\ge 1$ in what follows. We use the definition $h_{\la,\ell}\circ\rho$ to write
\[h_{\la,\ell}(\rho(\xi) )= m_{\la,\ell,1} (\xi) + m_{\la,\ell,2} (\xi) \] where
$m_{\la,\ell,1 } (\xi)=
\chi(\rho(\xi)) \frac{1}{2\pi} \iint \chi_1(u) (1-u)_+^\la \Psi(2^{-\ell}s) e^{is(\rho(\xi) -u) }\ud s\ud u.$
We have
\begin{multline} \F^{-1}[m_{\la,\ell,1}] (x) = \\ \frac{1}{(2\pi)^{d+1}} \iint \chi_1(u) (1-u)_+^\la \Psi(2^{-\ell}s) e^{-isu} \int \chi(\rho(\xi) )e^{is\rho(\xi) +i\inn{x}{\xi} } \ud \xi \ud s\ud u. \end{multline}
Analyzing the gradient of the phase function in the inner $\xi$ integral we get for $2^{\ell-1}<s<2^{\ell+1} $
\[ |s\nabla\rho(\xi) +x|\ge \begin{cases} |x|/2 \qquad &\text{ for } |x| \ge 2^{\ell+2} C_0
\\ 2^{-\ell-2} c_0 &\text{ for }|x|\le 2^{\ell-2} c_0 \end{cases}
\]
and an integration by parts in $\xi$ shows that for $2^{\ell-1}<s<2^{\ell+1} $
\Be \label{eq:xi-intbyparts}\Big| \int \chi(\rho(\xi) )e^{is\rho(\xi) +i\inn{x}{\xi} } \ud \xi \Big|
\lc_N\begin{cases} |x|^{-N-1} &\text{ for $|x|\ge 2^{\ell+2} C_0$} \\
2^{-\ell (N+1) } &\text{ for $|x|\le 2^{\ell-2} c_0$}
\end{cases}
\Ee
which after a trivial integration in $s$ and $u$ implies the desired bounds on $\cF^{-1} [m_{\la,\ell,1}](x)$.

We now examine $\cF^{-1}[m_{\la,\ell,2}]$ where
$m_{\la,\ell,2}(\xi) =h_{\la,\ell}(\rho(\xi)) -m_{\la,\ell,1} (\xi)$. The definition of of $m_{\la,\ell,2}$ involves a one-dimensional Fourier transform of $u\mapsto (1-\chi_1(u))(1-u)^\la_+ $, where the latter is supported in $(-\infty, 4/5)$. We perform a dyadic decomposition in the negative $u$ variables. Let $\eta_0\in C^\infty_c(\bbR)$ such that $\eta_0$ is supported in $(-5/6, 5/6)$ and $\eta_0(u)=1$ for $u\in (-4/5,4/5)$ and let, for $k\ge 1$, $\eta_k(u)=\eta_0(2^{-k}u)-\eta_0(2^{1-k}u)$. We then have
$m_{\la,\ell,2}=\sum_{k=0}^\infty m_{\la,\ell,2,k}$
where by integration by parts for all $N_1\ge 0$,
\begin{align*} m_{\la,\ell,2,k} (\xi) &=
\frac{\chi(\rho(\xi))}{2\pi} \iint \eta_k(u) (1-\chi_1(u)) (1-u)_+^\la \Psi(2^{-\ell}s) e^{is(\rho(\xi) -u) }\ud s\ud u
\\&=
\frac{\chi(\rho(\xi))}{2\pi} \iint \partial_u^{N_1}\big[ \eta_k(u)(1-\chi_1(u)) (1-u)_+^\la\big] \frac{\Psi(2^{-\ell}s) }{(is)^{N_1}} e^{is(\rho(\xi) -u) }\ud s\ud u
\end{align*} and the sum in $k$ converges rapidly in view of the estimate
\[|m_{\la,\ell,2,k} (\xi)|\lc 2^{k(\la-N_1)} 2^{\ell(1-N_1) }.\] Note that because of the cutoff $\chi(\rho(\xi))$ the same bound immediately holds for
$\|\cF^{-1}[ m_{\la,\ell,2,k} ] \|_\infty $. We will apply this with $N_1\gg N+\la$. After summing in $k$ we get a satisfactory bound for
$|x|\le C_0 2^{\ell+3}$. For $|x|\ge C_0 2^{\ell+2}$ we again integrate by parts in $\xi$ (\cf. \eqref{eq:xi-intbyparts}) and obtain the bound
$|\cF^{-1} [m_{\la,\ell,2,k}](x)|\lc 2^{k(\la-N_1) }2^{\ell(1-N_1) }(2^\ell |x|)^{-N-1} $ which again can be summed in $k$. Altogether we get
\Be\label{eq:FTofhlaell2} |\cF^{-1} [m_{\la,\ell,2}](x)| \lc_{N} 2^{-\ell N} (1+ 2^{-\ell}|x|)^{-N} \Ee for all $x\in \bbR^d$, which completes the proof.
\end{proof}

We get sharp estimates for the region $|x|\approx 2^\ell$ since $\partial \Om$ has nonvanishing Gaussian curvature everywhere.

\begin{lemma}\label{lem:main-kernelest}
$\|K_{\la,\ell}\|_\infty \lc 2^{-\ell( \la+\frac{d+1}2) }.$
\end{lemma}
\begin{proof} By Lemma \ref{lem:error-kernelest}
it suffices to prove the bound for $|x|\approx 2^\ell$.
We write the Fourier integral in $\rho$-polar coordinates $\xi=\varrho \xi'$ with $\xi'\in \partial \Omega$, $\ud \mu(\xi')=\inn{\fn(\xi')}{\xi'} d\sigma(\xi')$, where $\fn$ is the outer normal at $\xi'\in \partial \Om$.
We obtain
\begin{align*} (2\pi)^{d+1} K_{\la,\ell} (x)= & \iint (1-u)_+^\la \Psi(2^{-\ell}s) e^{-isu} \iint \varrho^{d-1} \chi(\varrho) e^{is\varrho +i\inn{x}{\varrho \xi'} } \ud \mu(\xi') \ud \varrho\ud s\ud u\\=
&c\int\Psi(2^{-\ell }s)s^{-\la-1} \int \varrho^{d-1} \chi(\varrho) e^{is(\varrho-1)} \int_{\partial \Om} e^{i\varrho \inn{x}{\xi'} } \ud\mu(\xi') \ud\varrho\ud s.
\end{align*}
Since $\partial \Omega$ has nonvanishing Gaussian curvature, the inner integral can be written, by the method of stationary phase, as a sum of two expressions of the form $
c_\pm e^{i \varrho \inn{x}{\xi'_\pm(x) }} a_\pm (\varrho,x)$ where $a_\pm$ are smooth and, together with their derivatives, satisfy the bound $O((1+|x|)^{-\frac{d-1}2})$. The points $\xi'_\pm(x)$ are the two unique points on $\partial \Om$ where $x$ is normal to $\partial\Om$.
Subsequent integration by parts in $\rho$ yields for $|s|\approx |x|\approx 2^\ell$
\[\Big| \int \varrho^{d-1} \chi(\varrho) e^{is(\varrho-1)} \int_{\partial\Om}e^{i\varrho \inn{x}{\xi'} } \ud\mu(\xi') \ud\varrho \Big|\lc 2^{-\ell(d+1)/2}
\] and after integrating in $s$ obtain the asserted bound.
\end{proof}

\subsection*{Stein--Tomas type estimates} For the proof of Theorem \ref{thm:STendpt} we need the following consequences of the Stein--Tomas restriction theorem. Note that \eqref{eq:mj-ST-est} corresponds to the $\VV(p,2)$ condition mentioned in (i), \S\ref{sec:known}.

\begin{lemma}\label{lem:ST-style} Let $1\le p\le \frac{2(d+1)}{d+3} $ and let $M$ be an integer with $M>d (\frac 1p-\frac 12)$.
Let $s \mapsto \vth_j(s)$ satisfy, for $\nu=0,1,\dots, M$,
\begin{equation}\label{eq:vartheta_j assumption}
\Big| \Big(\frac{\ud}{\ud s} \Big)^{ \nu} \big(\vth_j(s)) \Big|\le (1+|s|)^{-M} \,.
\end{equation}
Let $m_j(\xi)=\vth_j(2^j(1-\rho(\xi))$.
Then for each $j \geq 0$
\Be\label{eq:thetaj-ST-est} \Big \|\sum_{\substack{Q\in \fD_j}} 2^{j\frac{d+1}{2}} m_j(D) [f_{Q}\bbone_Q] \Big\|_2\lc \Big(\sum_{Q\in \fD_j} |Q| \|f_{Q} \|_{p}^{2}\Big)^{ \frac 12}.
\Ee
If, in addition, the functions $\vth_j$ are supported in $(1/4,4)$ then
\Be\label{eq:mj-ST-est} \Big \|\sum_{j\ge 0} \sum_{\substack{Q\in \fD_j}} 2^{j\frac{d+1}{2}} m_j(D) [f_{Q}\bbone_Q] \Big\|_2\lc \Big(\sum_{Q\in \fD} |Q| \|f_{Q} \|_{p}^{2}\Big)^{ \frac 12}.
\Ee

\end{lemma}
We omit the proof; it
relies on a standard argument by Fefferman and Stein \cite{FeffermanBR73}, with a refinement in \cite{Seeger-Indiana}.

\section{The main estimates}\label{sec:Mainestimates}

At the heart of the matter of the proof of Theorem \ref{thm:blackbox-sparse} lie certain estimates in Proposition \ref{CZ-prop} below in terms of collections of functions stemming from the Calder\'on--Zygmund decomposition.
To prove these estimates it is convenient to introduce a family of bilinear operators which allow an abstract formulation that is a priori unrelated to the Calder\'on--Zygmund decomposition.

In the following let $\fQ\subset \fD_{\ge 0}$.
On the set $\fQ$ we will consider the atomic measure given by $\mu( \{Q\}) =|Q|;$ i.e. for each subset $\fE\subset \fQ$ we have
\Be \label{eq:mudef} \mu( \fE) = \sum_{j\ge 0} 2^{jd} \#\fE_j \Ee where again $\fE_j$ is the subset of $\fE$ consisting of cubes of sidelength $2^j$. This choice of measure is natural since in the special case where $\fE$ is a {\it disjoint} collection of dyadic cubes $\mu(\fE)$ is just the Lebesgue measure of the union of the $Q$ in $\fE$.
Fix $\la$ and let $h_{\la, \ell}$ be defined as in \eqref{eq:hlaells}. Set
\Be\label{eq:Aldef} {A_{\la,\ell} f}= 2^{\ell(\la+\frac{d+1}{2}) } h_{\la,\ell}(\rho(D)) f.\Ee

For $\fQ\subset\fD $ and functions $\beta:\fQ\to \bbC$ we denote by $\ell^r(\fQ,\mu)$ the space of all $\beta$ such that
$\|\beta\|_{\ell^r(\mu)} =(\sum_{Q\in \fQ} |\beta(Q)|^r |Q|)^{1/r}$ and by $\ell^{r,1}(\fQ,\mu)$ the corresponding Lorentz space.
We also consider families of $L^p(\bbR^d)$ functions $F=\{F_Q\}_{Q\in \fQ} $ and set $\|F\|_{\ell^\infty (L^p)} =\sup_{Q\in \fQ} \| F_Q\|_p$. For any integer $s \geq 0$, define the
bilinear operators $\Xi_{s,\fQ} $ acting on $\ell^\infty(L^p(\bbR^d))\times \ell^{r,1}(\mu)$ by
\Be\label{eq:Xidef} \Xi_{s,
\fQ} [F,\beta] := \sum_{\ell\ge s} \sum_{Q\in \fQ_{\ell-s}} \beta(Q) A_{\la,\ell} [F_Q \bbone_Q].\Ee
The definition of $\Xi_{s,\fQ}$ depends on $\la$ via $A_{\la,\ell}$ in \eqref{eq:Aldef} but the operators $A_{\la,\ell}$ satisfy bounds that are uniform in $\la$ when $\la$ varies over a compact set, and this will also hold for the $\Xi_{s,\fQ} $. In analogy with the setup in \cite{HNS}, the normalization chosen in \eqref{eq:Aldef} is advantageous for standard interpolation arguments. To shorten the notation we use the following.

\begin{definition}
Let $\hypo$ denote the statement that $\VV(p,r)$ holds for all $p\in[\frac{2(d+1)}{d+3},p_\circ)$ and $r\in [p, r_*(p,p_\circ, r_\circ)$.
\end{definition}

Note that these correspond exactly to the hypothesis in Theorem \ref{thm:blackbox-sparse}.

\begin{thm} \label{thm:Xi}
Let $d\ge 2$, $\frac{2(d+1)}{d+3} \le p_\circ<\frac{2d}{d+1}$, $p_\circ \leq r_\circ \leq \frac{d-1}{d+1}p_\circ'$ and assume that $\hypo$
holds.
Then for $1\le p<p_\circ$ and
$p<r<\rstar$
there is $\eps=\eps(p,r)>0$ such that for all $s \geq 0$ and collections of disjoint cubes $\fQ\subset\fD_{\ge 0}$,
\[ \|\Xi_{s,\fQ} [F,\beta] \|_{L^r} \lc 2^{s(\frac dp-\eps)} \|\beta\|_{\ell^{r,1}(\mu)} \|F\|_{\ell^\infty(L^p)}. \]
\end{thm}
We will prove Theorem \ref{thm:Xi} in \S \ref{sec:ProofofMainProp}.
It will be convenient to also state a straightforward variant with larger cubes in $\fQ_{\ell+n}$, which is implied
by Theorem \ref{thm:Xi}.
\begin{cor} \label{cor:largecubescor} Assume the assumptions of Theorem \ref{thm:Xi} and let for $n\ge 0$,
\[\Xi_{-n,\fQ} [F,\beta] :=\sum_{\ell\ge 0} \sum_{Q\in \fQ_{\ell+n} } \beta(Q) A_{\la,\ell} [F_Q\bbone_Q].\]
Then for all $n\ge 0$, and collections of disjoint cubes $\fQ\subset\fD_{\ge 0}$,
\[ \|\Xi_{-n,\fQ}[F,\beta] \|_r\lc \|\beta\|_{ \ell^{r,1}(\mu)} \|F\|_ {\ell^\infty(L^p)}. \]
\end{cor}

\begin{proof}
We apply Theorem \ref{thm:Xi} for $s=0$. Indeed let for each cube $Q\in \fD$ denote by $R^n(Q)$ the unique cube in $\fD_{L(Q)+n}$ which contains $Q$. Let $\widetilde {\fQ}$ be the collection of all cubes $Q\in \fD_{\ge 0}$ such that $R^n(Q)\in \fQ$. If the cubes in $\fQ$ are disjoint then the cubes in $\widetilde \fQ$ are also disjoint.
For $Q\in \widetilde \fQ$ we set $\widetilde \beta(Q)=\beta(R^n(Q))$ and
$\widetilde F_Q=F_{R^n(Q)} $. Then
$\Xi_{-n, \fQ}(F,\beta) =\Xi_{0,\widetilde {\fQ}} (\widetilde F, \widetilde \beta)$, $\|\widetilde F\|_{\ell^\infty(\widetilde \fQ, L^p)}=\|F\|_{\ell^\infty(\fQ, L^p)}$ and
\[\|\widetilde \beta\|_{\ell^r(\mu, \widetilde \fQ)}^r =
\sum_{\ell\ge 0} \sum_{Q\in \widetilde \fQ_\ell}|Q| |\widetilde \beta(Q)|^r=
\sum_{\ell\ge 0} \sum_{Q'\in \fQ_{\ell+n} }\sum_{\substack{Q\in \fD_\ell\\Q\subset Q'} }|Q| |\beta(Q')|^r= \|\beta\|_{\ell^r(\mu, \fQ)}^r.
\]
The corollary now follows applying Theorem \ref{thm:Xi} to $\Xi_{0,\widetilde \fQ}[ \widetilde F,\widetilde \beta]$.
\end{proof}

The main motivation for Theorem \ref{thm:Xi} is its applications to the action of Bochner-Riesz type operators on the collection of functions in a Calder\'on--Zygmund decomposition.

\begin{prop}\label{CZ-prop}
Let $d\ge 2$, $\frac{2(d+1)}{d+3} \le p_\circ<\frac{2d}{d+1}$, $p_\circ \leq r_\circ \leq \frac{d-1}{d+1}p_\circ'$ and assume that $\hypo$ holds. Let $1\le p<p_\circ$ and $p<r<\rstar$. Let $\fQ\subset \fD_{\ge 0}$ be a collection of disjoint cubes, $\alpha>0$ and $\{f_Q\}_{Q \in \fQ}$ functions with $\supp(f_Q)\subset\overline{Q}$ and
\Be \label{eq:alphap-bds}
\int_Q |f_Q|^p \le \alpha^p|Q| \qquad \text{for all $Q\in \fQ$}.
\Ee
Then there exists an $\eps=\eps(p,r)>0$ such that for all $s \geq 0$, the inequality
\Be\label{eq:CZ-cons}
\Big\| \sum_{\ell\ge s} u_\ell h_{\la(p),\ell} (\rho(D)) \big[\sum_{Q\in \fQ_{\ell-s}} f_Q \big] \Big\|_r^r \lc \|u\|_{\ell^\infty}^r 2^{-\eps sr} \alpha^{r-p} \sum_{Q\in \fQ} \|f_Q\|_p^p
\Ee
holds for all sequences of complex numbers
$u=\{u_\ell\}_{\ell=0}^\infty$.
Moreover, for all $n\ge 0$,
\Be\label{eq:CZ-cons-n}
\Big\| \sum_{\ell\ge 0} u_\ell h_{\la(p),\ell} (\rho(D)) \big[\sum_{Q\in \fQ_{\ell+n}} f_Q \big ] \Big\|_r^r \lc \|u\|_{\ell^\infty}^r 2^{ndr/p} \alpha^{r-p} \sum_{Q\in \fQ} \|f_Q\|_p^p .
\Ee
\end{prop}

\begin{proof}[Proof (assuming Theorem \ref{thm:Xi})] Note $h_{\la(p),\ell}(\rho(D))=2^{-\ell d/p} A_{\la(p),\ell}$. For $Q\in \fQ$ set
\[ F_Q(x)= \begin{cases} f_Q/\|f_Q\|_p, &\text{if $\|f_Q\|_p\neq 0$}
\\ 0 &\text{otherwise} \end{cases} \quad \text{ and } \quad \beta(Q)= u_{L(Q)+s}2^{-L(Q)d/p} \|f_Q\|_p . \]
Then we get, with $\Xi_{s,\fQ}$ as in \eqref{eq:Xidef},
\[ \sum_{\ell\ge s} u_\ell h_{\la(p),\ell} (\rho(D)) \big[\sum_{Q\in \fQ_{\ell-s}} f_Q \big]= 2^{-sd/p} \Xi_{s,\fQ}[F,\beta]\,\]
{where we have of course used that $\ell=L(Q)+s$ for $Q \in \fQ_{\ell-s}$.} Applying Theorem \ref{thm:Xi} and the normalization $\|F\|_{\ell^\infty(L^p)}\le 1$ the
left-hand side of \eqref{eq:CZ-cons} is dominated by $[C 2^{-\eps s} \|\beta\|_{\ell^{r,1}(\mu) } ]^r$.
For $p<r<\infty$ the space $\ell^{r,1} $ is the real interpolation space $[\ell^\infty , \ell^p]_{\vartheta,1}$ with $\vartheta=p/r$ and therefore
\begin{align*}
\|\beta\|_{\ell^{r,1} (\fQ, \mu) } & \lc \|\beta\|_{\ell^{\infty} (\fQ, \mu) }^{1-\frac pr} \|\beta\|_{\ell^{p} (\fQ, \mu) }^{\frac pr}
\\ &\lc\|u\|_\infty \Big( \sup_Q \Big(\frac{1}{|Q|}\int_Q|f_Q|^p\Big)^{1/p} \Big)^{1-\frac pr}
\Big(\sum_{Q\in \fQ}\|f_Q\|_p^p\Big)^{1/r} \\
& \lesssim \|u\|_\infty \, \alpha^{ 1-\frac pr } \Big(\sum_{Q\in \fQ}\|f_Q\|_p^p\Big)^{1/r}
\end{align*}
using the assumption \eqref{eq:alphap-bds}. This establishes \eqref{eq:CZ-cons} and
\eqref{eq:CZ-cons-n} is obtained in the same way, using Corollary \ref{cor:largecubescor}.
\end{proof}

\section{Proof of Theorem \ref{thm:Xi}}\label{sec:ProofofMainProp}

\subsection{Reduction to a linear operator}
With $\fQ, \fQ_{j}$ as above and $A_{\lambda,\ell}$ defined as in \eqref{eq:Aldef}, let
\Be\label{eq:sAdef} \sA_{s, \fQ} F:= \sum_{\ell \ge s}\sum_{Q\in \fQ_{\ell-s}} A_{\la,\ell} [F_Q\bbone_Q]. \Ee
We will show that Theorem \ref{thm:Xi} is a consequence of the following.
\begin{thm} \label{thm:linearversion}
Let $d\ge 2$, $\frac{2(d+1)}{d+3} \le p_\circ<\frac{2d}{d+1}$, $p_\circ \leq r_\circ \leq \frac{d-1}{d+1}p_\circ'$ and assume that $\hypo$ holds.
Then for $1\le p<p_\circ$, $p<r<\rstar$
there is $\eps=\eps(p,r)>0$
such that for all $s\geq 0$, and all collections of disjoint cubes $\fQ\subset\fD_{\ge 0}$
\Be\label{eq:eps-pr-gain} \|\sA_{s,\fQ} F\|_r \lc 2^{s(\frac dp-\eps)} \mu(\fQ)^{1/r} \|F\|_{\ell^\infty(L^p)}.\Ee
\end{thm}

\begin{proof}[Proof of Theorem \ref{thm:Xi} assuming Theorem \ref{thm:linearversion}]
For completeness we include the standard argument (\cf. {\cite[Ch. V.3]{stein-weiss}}).
Let $\beta^*$ be the nonincreasing rearrangement of $\beta$. We may decompose $\beta= \sum_{k\in \bbZ} \beta^k$ where $\beta^k(Q)= \beta(Q) \bbone_{\fE^k} (Q) $ and $\fE^k=\{Q\in \fQ: \beta^*(2^{k+1} ) <|\beta(Q)| \le \beta^*(2^{k})\} . $ Observe that
$\mu(\fE^k)\le 2^{k+1} $, which also shows that the $\fE^k$ are finite sets. We have $\|\Xi_{s,\fQ}[F, \beta] \|_r\lc \sum_{k} \|\Xi_{s,\fQ}[F, \beta^k]\|_r $ which is written as
\[ \sum_{k}
\beta^*(2^k)\|\Xi_{s,\fQ}[F, \tfrac{\beta \bbone_{\fE^k}}{\beta^*(2^k) }]\|_r
= \sum_{k}\beta^*(2^k) \| \mathscr A_{s,\fE^k} F^k\|_r\]
where the function $G^k$ is defined by $F_Q^k(x)=F_Q( x) \bbone_{\fE^k}(Q) \frac{\beta(Q)}{\beta^*(2^k) }$. Note that $|F_Q^k(x)|\le |F_Q(x)|.$ Applying Theorem \ref{thm:linearversion} to $\sA_{s,\fE^k} F^k$ we get
\begin{align*} & \sum_{k}\beta^*(2^k) \| \mathscr A_{s,\fE^k} [ F^k]\|_r
\lc 2^{s(\frac dp-\eps)} \sum_{k}\beta^*(2^k) \mu(\fE^k)^{1/r} \|F^k\|_{\ell^\infty(L^p) }
\end{align*} and since $\mu(\fE^k)^{1/r}\lc 2^{k/r}$ and $ \|F^k\|_{\ell^\infty(L^p) }
\le \|F\|_{\ell^\infty(L^p) }$ we see that the right-hand side is $\lc 2^{s(\frac dp-\eps)} \|\beta\|_{\ell^{r,1}} \|F\|_{\ell^\infty(L^p) }$, as desired.
\end{proof}

The key to prove Theorem \ref{thm:linearversion} are the following propositions.

\begin{prop}\label{prop:p=p0}
Let $d\ge 2$, $\frac{2(d+1)}{d+3} \le p_\circ<\frac{2d}{d+1}$, $p_\circ \leq r_\circ \leq \frac{d-1}{d+1}p_\circ'$ and assume that $\hypo$ holds. Then for $\tfrac{2(d+1)}{d+3}\le p< p_\circ$, $p\le r<\rstar$ and all $s \geq 0$, and collections of disjoint cubes $\fQ\subset\fD_{\ge 0}$,
\begin{equation}\label{eq:p=p0}
\|\sA_{s,\fQ} F \|_{r} \lc 2^{s\frac{d}{p} } \mu(\fQ)^{ \frac{1}{r}} \|F\|_{\ell^\infty(L^{p} ) }.
\end{equation}
\end{prop}

\begin{prop} \label{prop:p=1}
Let $d \geq 2$. For all $1 < r < \infty$ there exists $\varepsilon(r)>0$ such that for all $s \geq 0$, and collections of disjoint cubes $\fQ\subset\fD_{\ge 0}$,
\begin{equation}
\label{eq:p=1}
\|\sA_{s,\fQ} F \|_{r} \lc 2^{s(d-\eps(r) )} \mu(\fQ)^{ \frac1{r}} \|F\|_{\ell^\infty(L^{1} )}.
\end{equation}
\end{prop}

We note that Proposition \ref{prop:p=p0} is essentially a re-statement of $\hypo$, and Proposition \ref{prop:p=1} is an improvement over the trivial
\Be\label{eq:tr-L1}\|\sA_{s, \fQ}F\|_{1} \lc 2^{sd} \mu(\fQ) \| F \|_{\ell^\infty(L^1)} \Ee
which follows since the $L^1\to L^1$ operator norm of $2^{\ell (\la+\frac{d+1}{2} )}h_{\la,\ell}(\rho(D))$ is $O(2^{\ell d})$. They will be proven in \S\ref{sec:p=p0} and \S\ref{sec:p=1} respectively.

Theorem \ref{thm:linearversion} now follows by a standard complex interpolation argument based on the interpolation formula
\[[ \ell^\infty (L^{u_0}),\ell^\infty(L^{u_1})]_\vartheta= \ell^\infty(L^u)\] with $(1-\vartheta)/u_0+\vartheta/u_1=1/u$ which holds if the $\ell^\infty$ norms are taken on a finite set (as in our applications).\footnote{One has to use the second $[\cdot,\cdot]^\theta$ method by Calder\'on in the general case.}
We first interpolate \eqref{eq:p=p0} for $p=\frac{2(d+1)}{d+3}$ and \eqref{eq:p=1} to obtain
\Be\label{eq:AsQ-p1}\|\sA_{s,\fQ} F\|_v \lc 2^{s(\frac du-\eps(u,v))} \mu(\fQ)^{\frac 1v} \|F\|_{\ell^\infty(L^{u})}, \text{ for $1<u<\tfrac{2(d+1)}{d+3}$, $ u < v < \tfrac{d-1}{d+1}u'$} \Ee for some $\eps(u,v)>0.$
Now fix $p$ and $r$ with $\frac{2(d+1)}{d+3} \le p<p_\circ$ and $p<r<\rstar$. We can then find pairs $(u,v)$ such that $1<u<\frac{2(d+1)}{d+3}$, $u < v < \frac{d-1}{d+1}u'$ and $(p_1,r_1)$ such that $\frac{2(d+1)}{d+3} \leq p_1 < p_\circ$, $p_1 \leq r_1 < r_*(p_1,p_\circ, r_\circ)$, with $(\frac{1}{p}, \frac {1}{r})$ in the open line segment connecting $(\frac{1}{u},\frac{1}{v})$ with $(\frac{1}{p_1},\frac{1}{r_1})$, i.e. $(\frac 1p, \frac 1r)=(1-\vth) (\frac{1}{p_1},\frac{1}{r_1})+ \vth (\frac{1}{u},\frac{1}{v})$ for some $\vth\in (0,1)$.\footnote{Here $(p_1,r_1)$ should be thought of being sufficiently close to $(p_\circ,r_\circ)$; note that $r_*(p_\circ, p_\circ, r_\circ)=r_\circ$.}
Now interpolate \eqref{eq:p=p0} for $(\frac 1{p_1}, \frac 1{r_1})$ with
\eqref{eq:AsQ-p1} for $(\frac{1}{u},\frac{1}{v})$. We then obtain \eqref{eq:eps-pr-gain} for the pair $(\frac 1p, \frac 1r)$ with $\eps(p,r):=\vth \eps(u,v)>0$.

\subsection{Proof of Proposition \ref{prop:p=p0}}\label{sec:p=p0}

As mentioned above, this is essentially a reformulation of $\hypo$ in which one replaced the normalized bumps $\chi(2^\ell(1-\varrho) )$ by $2^{\ell \la}h_{\ell,\la} (\varrho) $. The technical lemma that takes care of it is the following.
\begin{lemma} \label{lem:conseq-of-hyp}
Let $d \geq 2$, $1 \leq p \leq r < \infty$ and assume that $\VV(p,r)$ holds. Then for all $s \geq 0$,
\Be\label{eq:new-form-s}
\Big\|\sum_{\ell\ge s} A_{\la,\ell}
\big [ \sum_{Q\in \fD_{\ell-s}}f_{\ell,Q}\bbone_Q \big]\Big\|_r
\lc 2^{s\frac {d}{p}}
\Big(\sum_\ell \sum_{Q} |Q|\|f_{\ell,Q}\|_{p} ^{r} \Big)^{\frac 1r}\,.
\Ee
\end{lemma}

If $f_{\ell,Q} =F_Q$ for $Q\in \fQ_{\ell-s}$ (and $0$ otherwise) then the right-hand side in \eqref{eq:new-form-s} is clearly bounded by $2^{s \frac{d}{p}} \mu(\fQ)^{\frac{1}{r}} \| F \|_{\ell^\infty(L^p)}$, implying thus Proposition \ref{prop:p=p0}.

\begin{proof}

We first examine the case $s=0$.
Let $\eta\in C^\infty_c$ be supported in $(1/2,2)$ such that $\sum_{k\in \bbZ} \eta(2^k u)=1$ for $u>0$.
In view of the support of $h_{\lambda, \ell}$, decompose the convolution kernel of $A_{\la,\ell}$ using
\begin{subequations}\label{eq:thin-annuli}
\Be 2^{\ell (\lambda +\frac{d+1}2)} h_{\la,\ell}=\sum_{0 \leq m_1< \ell} \theta_{\la,\ell,m_1}+ \sum_{m_2>0} \widetilde \theta_{\la, \ell,m_2} \Ee where
\begin{align} \theta_{\la, \ell,m_1} (\varrho) &=
2^{\ell (\la +\frac{d+1}2 )}h_{\la,\ell} (\varrho)
\eta(2^{\ell-m_1} (1-\varrho) ),
\\
\widetilde \theta_{\la,\ell,m_2}(\varrho)&=
2^{\ell (\la+\frac{d+1}2 )} h_{\la,\ell}(\varrho)
\eta(2^{\ell+m_2} (1-\varrho)) .
\end{align}
\end{subequations}
This decomposition is done to exploit the hypothesis $\VV(p,r)$, since the $\theta_{\lambda, \ell,m_1}$ and $\widetilde{\theta}_{\lambda, \ell,m_2}$ are now compactly supported. Our goal is to show the inequalities
\begin{align} \label{eq:the-n-terms}
\Big\|
\sum_{\ell>m_1} \sum_{Q\in \fD_{\ell} } \theta_{\la,\ell,m_1} (\rho(D)) [f_{\ell,Q} \bbone_{Q} ]\Big\|_r \lc_{N}
2^{-m_1 (N+\la(r) )} \Big(\sum_\ell \sum_{Q\in \fD_\ell} 2^{\ell d} \|f_{\ell,Q} \|_{p}^{r}\Big)^{1/r}
\\
\label{eq:the-m-terms}
\Big\| \sum_{\ell>0} \sum_{Q\in \fD_\ell} \widetilde \theta_{\la,\ell,m_2} (\rho(D)) [g_{\ell,Q}\bbone_Q] \Big\|_r
\lc 2^{m_2(\la(p) -N_\circ)}
\Big(\sum_\ell \sum_{Q\in \fD_{\ell}} 2^{\ell d}
\|f_{\ell,Q} \|_{p}^ {r} \Big)^{1/r} .
\end{align}
Combining the estimates \eqref{eq:the-n-terms} and \eqref{eq:the-m-terms} (and recalling that $N_\circ>\la(p) $)
yields \eqref{eq:new-form-s} for $s=0$.

Using Lemma \ref{lem:ptwisebdhellla} we can write
\begin{subequations}
\begin{align} \theta_{\la,\ell,m_1}(\varrho) &=c_{N,\la} 2^{-m_1 N} 2^{ \ell \frac{d+1}2 }\chi_{\la,\ell,m_1} (2^{\ell-m_1} (1-\varrho)) ,
\\
\widetilde \theta_{\la,\ell,m_2} (\varrho)&=\widetilde c_{N_\circ,\la} 2^{-m_2N_\circ} 2^{ \ell \frac{d+1}2 } \widetilde \chi_{\la,\ell,m_2} (2^{\ell+m_2} (1- \varrho)),
\end{align} for suitable $\chi_{\la,\ell,m_1}, \,\widetilde \chi_{\la,\ell,m_2} \in \cY_M$.
\end{subequations}

For each $R'\in \fD_{\ell-m_1} $ we let $Q(R')$ be the unique $Q\in \fD_\ell$ that contains $R'$.
Writing $\bbone_Q=\sum_{R'\subset Q} \bbone_{R'} $ we then have
\begin{multline*}
\Big \| \sum_{\ell>m_1} \sum_{Q\in \fD_{\ell} } \theta_{\la,\ell,m_1} (\rho(D)) [f_{\ell,Q} \bbone_{Q} ] \Big\|_r
= c_{N,\la} 2^{-m_1 (N-\frac{d+1}{2} )} \times\\
\Big\|
\sum_{\ell>m_1} \sum_{Q\in \fD_\ell} \sum_{\substack {R'\in\fD_{\ell-m_1}}} 2^{(\ell-m_1)\frac{d+1}{2}}\chi_{\la,\ell,m_1} ( 2^{\ell-m_1} (1-\rho(D)) )[f_{\ell-m_1,Q(R')} \bbone_{R'} ]\Big\|_r
\end{multline*} and we can use the hypothesis $\VV(p,r)$ to bound the expression on the right-hand side by a constant times
\begin{align*}
&2^{-m_1 (N - \frac{d+1}{2}) } \Big(\sum_{\ell>m_1} \sum_{Q\in \fD_\ell} \sum_{\substack{R'\in \fD_{\ell-m_1} \\R'\subset Q}}
\big[ 2^{(\ell-m_1) d/r }\|f_{\ell,Q} \bbone_{R'} \|_{p} \big]^{r}\Big)^{1/r}
\\
&
\lc
2^{-m_1 (N +\frac dr- \frac{d+1}{2}) } \Big(\sum_{\ell} \sum_{Q\in \fD_\ell}
\big[ 2^{\ell d/r }\|f_{\ell,Q} \bbone_{Q} \|_{p} \big]^{r}\Big)^{1/r}
\end{align*}
where we have used $r\ge p$,
$\sum_{R'\in \fD_{\ell-m_1} }\|f_{\ell,Q} \bbone_ {R'} \|_{p}^r\le \|f_{\ell,Q}\|_{p}^r $ for all $Q\in \fD_\ell$. This finishes the proof of \eqref{eq:the-n-terms}.

We now turn to the proof of \eqref{eq:the-m-terms}.
To apply Lemma \ref{lem:ptwisebdhellla} we label $j=\ell+m_2$, and set, for $R'\in \fD_j$,
\[ g^{m_2}_{j,R'} =
\sum_{ \substack{Q\in \fD_{j-m_2}: Q\subset R' } }f_{j-m_2,Q} \bbone_Q. \]
Then
\begin{align*}
\Big\| \sum_{\ell>0} \sum_{Q\in \fD_\ell} & \widetilde \theta_{\la,\ell,m_2} (\rho(D)) [f_{\ell,Q}\bbone_Q]
\Big\|_r \\
&
=2^{-m_2(N_\circ+\frac{d+1}2)}\Big\| \sum_{j>m_2} \sum_{R'\in \fD_j} 2^{j\frac{d+1} 2} \widetilde \chi_{\la,j-m_2,m_2}(2^j(1-\rho(D))) g^{m_2}_{j,Q} \Big\|_r
\\&\lc 2^{ -m_2(N_\circ +\frac{d+1}2)} \Big(\sum_{j>m_2} \sum_{R'\in \fD_j} \big [2^{jd} \big\|
g^{m_2}_{j,R'}
\big\|_{p} \big]^r \Big)^{\frac 1r}
\end{align*}
where we applied the hypothesis $\VV(p,r)$ to get the bound in the third line.
Now for $j>m_2$,
\begin{align*}
\Big(\sum_{R'\in \fD_j} 2^{jd} \big\|
g^{m_2}_{j,R'}
\big\|_{p}^r\Big)^{\frac 1r}
&=\Big(\sum_{R'\in \fD_j} 2^{jd} \Big( \sum_{\substack {Q\in \fD_{j-m_2}\\Q\subset R'}} \|f_{j-m_2,Q} \|_{p} ^{p} \Big)^{\frac r{p}} \Big)^{\frac 1r}
\\
&\lc 2^{m_2 d/p} \Big(2^{ (j-m_2)d}\sum_{Q\in \fD_{j-m_2} } \|f_{j-m_2,Q} \|_{p} ^{r} \Big)^{\frac 1r} ,
\end{align*}
by H\"older's inequality in the inner $Q$-sum.
Combining the above
we get \eqref{eq:the-m-terms}.

Finally we consider the case $s>0$. Define
for $R'\in \fD_\ell$, $F_{\ell,R'}=\sum_{Q\subset R'} f_{\ell-s,Q}$ where the sum is taken over the cubes in
$\fD_{\ell-s} $ which are subcubes of $R'$. Note that
\[\big\| F_{\ell, R'} \|_{p} = \Big(\sum_{\substack {Q\in \fD_{\ell-s}\\ Q\subset R'}} \|f_{\ell-s,Q} \|_{p}^{p} \Big)^{\frac 1{p}}
\le 2^{sd(\frac 1{p}-\frac 1r)}
\Big(\sum_{\substack {Q\in \fD_{\ell-s}\\ Q\subset R'}} \|f_{\ell-s,Q} \|_{p}^{r} \Big)^{\frac 1{r}}.
\] Applying the result for $s=0$ proved above to the family of functions
$\{F_{\ell, R'}\}$ we get that the left-hand side of \eqref{eq:new-form-s} is dominated by a constant times
\begin{align*}
&\Big(\sum_{\ell\ge s} \sum_{R'\in \fD_\ell} 2^{\ell d} \|F_{\ell,R'} \|_{p} ^r\Big)^{1/r}
\lc 2^{ sd/{p} } \Big(\sum_{\ell\ge s} \sum_{Q\in \fD_{\ell-s}} 2^{ (\ell-s) d} \|f_{\ell-s,Q} \|_{p} ^r\Big)^{1/r}
\end{align*}
and we get \eqref{eq:new-form-s} for all $s\ge 0$.
\end{proof}

\subsection{Proof of Proposition \ref{prop:p=1}} \label{sec:p=1}
It follows from the inequalities
\begin{align}
\label{eq:p=1-r=2}
\|\sA_{s,\fQ}F\|_2 &\lc 2^{s(3d+1)/4} \mu(\fQ)^{1/2}
\|F\|_{\ell^\infty(L^1)}
\\
\label{eq:p=1--r}
\|\sA_{s, \fQ} F\|_{r_1} &\lc 2^{sd} \mu(\fQ)^{1/r_1}\|F\|_{\ell^\infty(L^1)}, \quad 1\le r_1<\infty.
\end{align}
Indeed, if $2\le r<\infty$ we choose $r_1>r$ large in \eqref{eq:p=1--r} and obtain \eqref{eq:p=1} by taking a mean of \eqref{eq:p=1-r=2} and \eqref{eq:p=1--r}.
Similarly (but less interesting for our purpose) one gets \eqref{eq:p=1} for $1<r\le 2$ by taking a mean of \eqref{eq:p=1-r=2} and \eqref{eq:tr-L1}. We note that our argument for \eqref{eq:p=1-r=2} does not use the disjointness property of the family of cubes $\fQ$, but the argument for \eqref{eq:p=1--r} strongly relies on it.

\subsubsection{The case $p=1$, $r=2$: proof of \eqref{eq:p=1-r=2}}
We will first formulate a version of \eqref{eq:p=1-r=2} for linear combinations of radial bump multipliers $\chi(2^\ell(1-\rho))$, and then subsequently replace the radial bumps by the multipliers $2^{\ell\la} h_{\la,\ell} \circ\rho$ to get \eqref{eq:p=1-r=2}.

\begin{lemma}\label{lem:CSnew} Let $d \geq 2$ and $\{\chi_j\}_j \subseteq \cY_M$ for large $M\gg 10d$.
For all $s\ge 0$,
\Be \label{eq:CS1}\Big \|\sum_{j\ge 2s} 2^{j\frac{d+1}2}
\chi_j(2^j(1-\rho(D))) \big[\sum_{Q\in\fQ_{j-s} } F_Q\bbone_Q\big] \Big \|_2 \lc 2^{s \frac{d+1}2} \mu(\fQ)^{\frac 12} \| F \|_{\ell^\infty (L^1)}
\Ee
holds for all finite $\fQ \subset \fD_{\ge 0}$.
Moreover, for $j\ge 0$, $0<L\le j/2$
\Be\label{eq:CS2} \Big \| 2^{j\frac{d+1}2} \chi_j(2^j(1-\rho(D))) \big[ \sum_{Q\in\fQ_{L}} F_Q\bbone_Q\big] \Big \|_2 \lc 2^{j\frac{3d+1}{4} -Ld}
\mu(\fQ)^{\frac 12} \| F \|_{\ell^\infty (L^1)}.
\Ee
\end{lemma}
An immediate corollary (unifying and slightly weakening \eqref{eq:CS1}, \eqref{eq:CS2}) is

\begin{cor}
\label{cor:CSnew}
For $\ka\ge 0$ we have \Be \label{eq:CS3new}\Big \|\sum_{j\ge \ka} 2^{j\frac{d+1}2}
\chi_j(2^j(1-\rho(D))) \big[\sum_{Q\in\fQ_{j-\ka} } F_Q \bbone_Q\big] \Big \|_2 \lc 2^{\ka\frac{ 3d+1}4} \mu(\fQ)^{\frac 12} \| F \|_{\ell^\infty (L^1)}.
\Ee
\end{cor}
\begin{proof}
The term $\sum_{j\ge 2\ka} $ is handled using \eqref{eq:CS1} which gives the better $L^2$-bound $2^{\ka\frac{d+1}{2}}\mu(\fQ)^{1/2}$.
For $\sum_{\ka\le j< 2\ka}$ we apply \eqref{eq:CS2} with $L=j-\ka\in [0,j/2]$. By Minkowski's inequality the resulting $L^2$ bound is $ \sum_{\ka\le j\le 2\ka} 2^{\ka d -j\frac{d-1}{4} } \mu(\fQ)^{\frac 12} \| F \|_{\ell^\infty (L^1)} \lc 2^{\ka\frac{3d+1}{4} } \mu(\fQ)^{\frac 12} \| F \|_{\ell^\infty (L^1)}$. This gives \eqref{eq:CS3new}.
\end{proof}
\begin{proof}[Proof of Lemma \ref{lem:CSnew}]
Assume, without loss of generality, that $\| F \|_{\ell^\infty (L^1)} \leq 1$. We use arguments by Christ--Sogge \cite{ChristSogge1, ChristSogge2}; these do not require a curvature assumption on $\partial\Om$.
One can decompose
\Be \label{Cordoba-dec}
\chi_j(2^j(1-\rho(\xi)))=\sum_{\nu} \chi_{j,\nu}(\xi)
\Ee where the sum in $\nu$ is extended over an index set $\cI_j$ of cardinality $O(2^{j(d-1)/2}) $. Each multiplier $\chi_{j,\nu}$ is supported in a $(2^{-j}, 2^{-j/2}, \dots, 2^{-j/2})$ box essentially tangential to $\partial\Om$. Moreover, the supports of $\chi_{j,\nu}$ have bounded overlap in the sense that $\sum_{j,\nu}|\chi_{j,\nu}(\xi)|\lc 1$, and we have the kernel estimates
\begin{multline} \label{eq:nonisotropicnew} |\cF^{-1}[\chi_{j,\nu} ] (x)|+ |\cF^{-1}[|\chi_{j,\nu} |^2] (x)|\\
\lc K_{j,\nu}(x):=2^{-j(d+1)/2} (1+2^{-j} |\inn {x}{e_{j,\nu} }|)^{-N_1}
(1+2^{-j/2} |P_{j,\nu}^\perp (x)|)^{-N_2};
\end{multline}
here $e_{j,\nu}$ is a unit vector orthogonal to the surface $\partial\Om$ on a point in $\supp (\chi_{j,\nu})$ and $P_{j,\nu}^\perp $ is the orthogonal projection to the hyperplane orthogonal to $e_{j,\nu}$ and $N_1, N_2\le M$ (and by choosing $M$ large enough we may assume that $N_1>1$, $N_2>d-1$).

By orthogonality (due to the bounded overlap condition) we have
\begin{multline} \label{eq:use-ortho1new}
\Big \|\sum_{j>2s} 2^{j\frac{d+1}2}
\chi_j(2^j(1-\rho(D))) \big[\sum_{Q\in\fQ_{j-s} } F_Q\big] \Big \|_2 \\ \lc \Big(
\sum_{j>2s} 2^{j(d+1) }\sum_{\nu \in \cI_j}\Big\|
\chi_{j,\nu} (D) \big[\sum_{Q\in\fQ_{j-s} } F_Q\big] \Big \|_2^2\Big)^{1/2}
\end{multline}
and, similarly for every $j$, $L\le j/2$,
\begin{multline} \label{eq:use-ortho2new} \Big \| 2^{j\frac{d+1}2} \chi_j(2^j(1-\rho(D))) \big[\sum_{Q\in\fQ_{L}} F_Q\big] \Big \|_2^2 \\
\lc 2^{j(d+1)/2} \Big(\sum_{\nu\in \cI_j}
\Big \| \chi_{j,\nu}(D) \big[ \sum_{Q\in \fQ_L} F_Q\big] \Big \|_2^2 \Big)^{1/2}.
\end{multline}
We claim that for fixed $j\ge 0$, $\nu\in \cI_j$
\begin{align} \label{eq:CS1nu}&\Big \|\chi_{j,\nu} (D)\big[\sum_{Q\in\fQ_{j-s}} F_Q\big]\Big \|_2 \lc 2^{s\frac{d+1}{2} -j\frac{3d+1}{4} } \mu(\fQ_{j-s})^{\frac 12}, \quad& \text{ $s\le j/2$, }
\\ \label{eq:CS2nu}
&\Big \|\chi_{j,\nu}(D)\big[\sum_{Q\in\fQ_L} F_Q\big]\Big \|_2\lc 2^{-Ld} \mu(\fQ_L)^{\frac 12}, &\text{ $L\le j/2$}
\end{align}
and then the inequality \eqref{eq:CS1} follows from \eqref{eq:use-ortho1new} and \eqref{eq:CS1nu}, together with the bound $\#\cI_j\lc 2^{j(d-1)/2}$ and $\sum_j \mu(\fQ_{j-s})\le \mu(\fQ)$. Likewise
\eqref{eq:CS2} follows from \eqref{eq:use-ortho2new} and \eqref{eq:CS2nu}.

It remains to prove \eqref{eq:CS1nu} and \eqref{eq:CS2nu}.
For \eqref{eq:CS1nu} we use $\| F \|_{\ell^\infty(L^1)} \leq 1$ and write
\begin{align*} \Big\|
\chi_{j,\nu} (D) \big[\sum_{Q\in\fQ_{j-s} } F_Q\big] \Big \|_2^2 &= \iint \cF^{-1}[|\chi_{j,\nu}|^2] (x-y ) \sum_{Q\in\fQ_{j-s} } F_Q(y) \,\ud y
\sum_{Q\in\fQ_{j-s} } \overline{F_Q(x)} \,\ud x
\\ &\lc \#\fQ_{j-s} \, \sup_x \int K_{j,\nu}(x-y) \sum_{Q\in \fQ_{j-s}} |F_Q(y)|\ud y
\end{align*} with $K_{j,\nu}$ as in \eqref{eq:nonisotropicnew}.
For $x\in \bbR^d$ and $n_1,n_2 >0$, define the regions
\begin{align*}
\cR_{j,\nu,s}^{0,0}(x)&:=\{ y \in \R^d: |\inn{x-y}{e_{j,\nu} }| \le 2^j, \, |P^\perp_{j,\nu}(x-y)| \le 2^{j-s}\},
\\
\cR_{j,\nu,s}^{n_1,0}(x)&:=\{ y \in \R^d: 2^{j+n_1-1} \le |\inn{x-y}{e_{j,\nu}} | \le 2^{j+n_1} , \, |P^\perp_{j,\nu}(x-y)| \le 2^{j-s}\},
\\
\cR_{j,\nu,s}^{0,n_2}(x)&:=\{ y \in \R^d: |\inn{x-y}{e_{j,\nu}} | \le 2^{j} , \, 2^{j-s+n_2-1}\le |P^\perp_{j,\nu}(x-y)| \le 2^{j-s+n_2} \},
\\
\cR_{j,\nu,s}^{n_1,n_2}(x)&:=\{ y \in \R^d: 2^{j+n_1-1} \le |\inn{x-y}{e_{j,\nu}} | \le 2^{j+n_1} , \\
& \hspace{5cm} \, 2^{j-s+n_2-1}\le |P^\perp_{j,\nu}(x-y)| \le 2^{j-s+n_2}\}.
\end{align*}

Observe that $\#\fQ_{j-s} \lc 2^{(s-j)d} \mu(\fQ_{j-s})$. Moreover, for $s\le j/2$ we have for all $n_1,n_2 \geq 0$
\begin{align}
\label{eq:Kjnubound}
&\sup_x \sup_{y\in \cR_{j,\nu,s}^{n_1,n_2}(x) }|K_{j,\nu} (x-y)| \le C_{N_1,N_2} 2^{-j\frac{d+1}{2} } 2^{-n_1N_1- (n_2+\frac j2-s)N_2}\,, \\
\label{eq:cardcubesbd}
&\sup_x \#\{ Q\in \fQ_{j-s} : Q\cap \cR_{j,\nu,s}^{n_1,n_2}(x)\neq\emptyset\} \lc 2^{s+n_1+n_2(d-1)}\,.
\end{align}
Combining these observations and summing in $n_1,n_2 \geq 0$ yields \eqref{eq:CS1nu}.

In order to prove \eqref{eq:CS2nu} we argue similarly. For $L<j/2$
we get as above
\begin{align*} \Big\|
\chi_{j,\nu} (D) \big[\sum_{Q\in\fQ_{L} } F_Q\big] \Big \|_2^2 \lc \#\fQ_{L} \, \sup_x \int K_{j,\nu}(x-y) \sum_{Q\in \fQ_{L}} |F_Q(y)|\ud y.
\end{align*}
Now
use $\#\fQ_{L} \lc 2^{-Ld} \mu(\fQ_{L})$, \eqref{eq:Kjnubound} with $s=j/2$ and the estimate
\begin{equation*}
\sup_x \#\{ Q\in \fQ_{L} : Q\cap \cR_{j,\nu,j/2}^{n_1,n_2}(x)\neq\emptyset\} \lc 2^{-Ld} 2^{j \frac{d+1}{2}}2^{s+n_1+n_2(d-1)}.
\end{equation*}
This leads to \eqref{eq:CS2nu}.
\end{proof}

We next show how to replace the normalized bumps in Corollary \ref{cor:CSnew} by $2^{\ell \lambda} h_{\ell, \lambda}(\rho)$ to obtain \eqref{eq:p=1-r=2}. The argument is very similar to that in Lemma \ref{lem:conseq-of-hyp}.

\begin{proof}[Proof of \eqref{eq:p=1-r=2}]
Assume, without loss of generality, that $\| F \|_{\ell^\infty (L^1)} \leq 1$.
We decompose as in \eqref{eq:thin-annuli} and write
$2^{\ell(\lambda+\frac{d-1}{2})}h_{\lambda,\ell}(\varrho)$ as
\[C_M 2^{\ell\frac{d+1}{2} }\Big[ \sum_{m_1\le \ell} 2^{-m_1N} \chi_{\lambda, \ell, m_1}(2^{\ell-m_1} (1-\varrho)) +\sum_{m_2>0} 2^{-m_2N_\circ} \widetilde \chi_{\lambda,\ell,m_2}(2^{\ell+m_2}(1-\varrho))\Big]\]
with $\chi_{\lambda, \ell, m_1},\widetilde\chi_{\lambda, \ell, m_2}\in \cY_M$ and $M>100d$.
We then bound $\|\sA_{s,\fQ} F\|_2$ by a constant times
\begin{equation}\label{eq:dec I II III}
\sum_{m_1> s} 2^{-m_1( N-\frac{d+1}{2})} I_{m_1}+
\sum_{m_1=0}^s 2^{-m_1(N-\frac{d+1}{2})} II_{m_1}+
\sum_{m_2\ge 0} 2^{-m_2(\frac{d+1}{2}+N_\circ)} III_{m_2}
\end{equation}
where
\begin{align*}
I_{m_1}&= \Big\|\sum_{\ell\ge m_1 } \sum_{R\in \fD_{\ell-m_1} } \sum_{Q\in \fQ_{\ell-s}} 2^{(\ell-m_1) \frac{d+1}{2}} \chi_{\la,\ell,m_1}(2^{\ell-m_1} (1- \rho(D))) [F_{Q} \bbone_{R\cap Q} ] \Big\|_2,
\\
II_{m_1}&= \Big\|\sum_{\ell\ge s }
2^{(\ell-m_1) \frac{d+1}{2}} \chi_{\la,\ell,m_1}(2^{\ell-m_1} (1- \rho(D))) [ \sum_{Q\in \fQ_{\ell-s}} F_{Q} \bbone_Q] \Big\|_2,
\\
III_{m_2}&=\Big\|\sum_{\ell\ge s }
2^{(\ell+m_2) \frac{d+1}{2}} \widetilde \chi_{\la,\ell,m_2}(2^{\ell+m_2} (1- \rho(D))) [ \sum_{\substack{ Q\in \fQ_{\ell-s} }} F_{Q} \bbone_Q] \Big\|_2.
\end{align*}

Let $m_1>s$ and $Q^{m_1-s}(R)$ be the unique dyadic cube with sidelength $2^{L(R)+m_1-s} $ containing $R$. Let $\fR^{m_1-s}(\fQ)$ be the family of all $R\in \fD$ such that $L(R)\ge 0$ and such that $Q^{m_1-s}(R)$ belongs to $\fQ$.
Parametrizing $j=\ell-m_1$ the term $I_{m_1}$ can be rewritten as \[\|\sum_{j\ge 0 } \sum_{R\in \fR^{m_1-s}_j (\fQ) } 2^{j \frac{d+1}{2}} \chi_{\la,j+m_1,m_1}(2^j (1- \rho(D))) [f_R \bbone_{R} ] \|_2, \text{ with } f_R:= F_{Q^{m_1-s}(R)}. \] We now apply Corollary \ref{cor:CSnew} with $\kappa=0$ and note that $\mu(\fR^{m_1-s}_j(\fQ))=\mu(\fQ_{j+m_1})$. Since $\|f_R\|_1\le 1$, we obtain $I_{m_1} \lc \mu(\fQ)^{1/2}$ and thus the first term on the right-hand side of \eqref{eq:dec I II III} is bounded by $C\mu(\fQ)^{1/2}$, which is a better bound.

For the terms $II_{m_1}$ we have $s\ge m_1$.
Changing the summation variable to $j=\ell-m_1$ one can apply
Corollary \ref{cor:CSnew} with $\ka=s-m_1$ to get
$II_{m_1} \lc 2^{(s-m_1)\frac{3d+1}{4}} \mu(\fQ)^{1/2} $
Similarly for $III_{m_2}$, changing the summation variable to $j=\ell+m_2$ we see that
Corollary \ref{cor:CSnew} with $\ka=s+m_2$ yields the bound $III_{m_2} \lc 2^{(s+m_2)\frac{3d+1}{4}} \mu(\fQ)^{1/2} $. After summing we bound the
second and third terms on the right-hand side of \eqref{eq:dec I II III} both by $C2^{s(3d+1)/4} \mu(\fQ)^{1/2} $.
\end{proof}

\subsubsection{The case $p=1$, $r> 2$: proof of \eqref{eq:p=1--r}}
Since the inequality has already been proved (in fact improved) for $r\le 2$ we focus on the case for large $r$, and by interpolation it suffices to assume that $r>2$ is an integer.
We now rely on the kernel estimates
in \S\ref{sec:decomp}. We prove straightforward size estimates which are close to an argument used by Conde-Alonso, Culiuc, Di Plinio, and Ou \cite{conde-alonso-etal} in the analysis of rough singular integral operators.

We let $K_\ell=\cF^{-1}[ 2^{\ell(\frac{d+1}{2}+\la)} h_{\la,\ell}\circ\rho]$, the convolution kernel of the operator $A_{\la,\ell}$.
From Lemmas \ref{lem:error-kernelest} and \ref{lem:main-kernelest} we have
the kernel estimates
$|K_{\ell}(x)|\lc 1$ for $|x|\approx 2^{\ell}$,
$2^{-\ell d}|K_{\ell}(x)|\lc c_N 2^{-\ell N} $ for $|x|\le c_\circ 2^{\ell}$, and $2^{-\ell d} |K_{\ell}(x)|\lc\widetilde c_N|x|^{-N} $ for $|x|>C_\circ 2^\ell $.
We shall only use the slightly weaker bound
\Be\label{eq:Hlndef} 2^{-\ell d} |K_{\ell} (x)|\lc \sum_{n\ge 0} 2^{-n(N-d)} H_{\ell,n}(x) \quad \text{ with } H_{\ell,n}(x)=2^{-(\ell+n) d}\bbone_{\{|x|\le 2^{\ell+n}\}},\Ee where $N>d$. These favorable $L^\infty$ bounds are crucial for our argument; if we were to replace $2^{\ell \la}h_{\ell,\la}(\varrho)$ with $\chi(2^\ell(1-\varrho))$ for generic $\chi\in \cY_M$ they would no longer hold. Assume, without loss of generality, that $\| F \|_{\ell^\infty (L^1)} \leq 1$.
By \eqref{eq:Hlndef}
\[ \|\sA_{s,\fQ} F\|_r \lc \sum_{n=0}^\infty 2^{-n(N-d)} \Big \|\sum_{\ell>s} 2^{\ell d} H_{\ell,n} *\sum_{Q\in \fQ_{\ell-s} }|F_Q|\bbone_Q \Big \|_r.\] Setting
$G_{ L} : =2^{Ld} \sum_{Q\in \fQ_{L} } |F_Q |\bbone_Q $
the inequality $\|\sA_{s,\fQ} F\|_r \lc 2^{sd} \mu(\fQ)^{1/r} $ follows
from the bound
\Be \label{eq:no-cancellationsG}
\Big\|\sum_{\ell>s} H_{\ell,n} *G_{\ell-s} \Big\|_r^r \lc_{r} \mu(\fQ).
\Ee
Since $r$ is an integer we have that
the left-hand side in \eqref{eq:no-cancellationsG} is bounded by
\begin{equation}\label{eq:no-cancellations1-2}
C r!\sum_{\ell_1\ge\ell_2\ge \dots\ge \ell_r} \,
\int\limits_{ \substack{(y^1,\dots, y^r) \\ \in (\bbR^d)^r}} \int \prod_{i=1}^r \Big[H_{\ell_i,n} (x-y^{i} ) G_{\ell_i - s}(y^i) \Big] \ud x
\ud (y^1,\dots, y^r) .
\end{equation}

Observe that if there is an $x$ such that $\prod_{i=1}^r H_{\ell_i,n} (x-y^{i} ) \neq 0$
then we have $|y^i-y^{i+1}| \le 2^{\ell_{i}+n+1}$ for $i=1,\dots, r-1$. In this situation we also have the identity
$H_{\ell_i,n}(x-y^i) = H_{\ell_i,n}( \frac{ y^{i}-y^{i+1}}{2})$ for $1\le i\le r-1$; in addition, $\int H_{\ell_r,n} (x-y^r) \ud x\lc 1.$ We use these pointwise estimates and integrate in $x$ first
to bound \eqref{eq:no-cancellations1-2} by a constant times
\begin{multline}
\label{eq-no-cancellations2}
\int_{ (\bbR^d)^r} \sum_{\ell_1} G_{\ell_1-s}(y^1) \prod_{i=1}^{r-1} \Big[ \sum_{\ell_{i+1}=0}^{\ell_i}
H_{\ell_i,n}( \tfrac{ y^{i}-y^{i+1}}{2}) G_{\ell_{i+1}-s}(y^{i+1}) \Big]
\ud y^r \dots \ud y^1.
\end{multline}
For fixed $y^i$, with $1\le i\le r-1$ we have
\begin{align*}
&\int\limits_{y^{i+1} \in \bbR^d}
\sum_{\ell_{i+1}=0}^{\ell_i}
H_{\ell_i,n}( \tfrac{ y^{i}-y^{i+1}}{2}) G_{\ell_{i+1}-s}(y^{i+1}) \ud y^{i+1}
\\ &\lc 2^{-(\ell_i+n)d}\int\limits_{|y^{i+1}-y^i|\le 2^{\ell_i+n+1} }\sum_{\ell_{i+1}=0}^{\ell_i} 2^{(\ell_{i+1}-s)d} \sum_{Q\in \fQ_{\ell_{i+1}-s}} |F_Q(y^{i+1} ) \bbone_Q(y^{i+1})| \ud y^{i+1}
\\
&\lc 2^{-(\ell_i+n)d} \sum_{\ell_{i+1}=0}^{\ell_i}\sum_{ \substack{ Q\in \fQ_{\ell_{i+1}-s}\\ \dist(Q,y^i)\le 2^{\ell_i+n+1}}} |Q|\int |F_Q(y^{i+1}) |\ud y^{i+1} \lc 1
\end{align*}
where we used that the cubes in $\fQ $ are disjoint and the $F_Q$ have normalized $L^1$ norm. Thus integrating in \eqref{eq-no-cancellations2} first in $y^r$, then in $y^{r-1}$, and so on, we obtain
\begin{align*}
\Big\|\sum_{\ell>s} H_{\ell,n} *G_{\ell-s} \Big\|_r^r &\lc_r \int\limits_{y^1\in \bbR^d } \sum_{\ell_1}
G_{\ell_1-s}(y^1)
\ud y^1
\\
&\lc_r \sum_{\ell_1}
\sum_{Q\in \fQ_{\ell_1-s}} |Q| \|F_Q\|_1 \,
\lc \mu(\fQ)
\end{align*} and \eqref{eq:no-cancellationsG} is proved. This finishes the proof of \eqref{eq:p=1--r}. \qed

\color{black}
\section{Sparse domination, Part I} \label{sec:sparse-partI} Here we prove Theorem \ref{thm:blackbox-sparse}.
Without loss of generality, we will assume that $q < p'$; the less interesting sparse bound $(p,q_1)$ with $q_1 \geq p'$ would be implied by Hölder's inequality from any $(p,q)$ with $q < p'$. In what follows we assume that $\nc$ is a fixed positive integer as in \eqref{eq:ncircdef}; in particular, this implies $\nc\ge 5$. Implicit constants are allowed to depend on $\nc$.
Define the modified functionals (allowing averages over triple cubes for the functions $|f_2|^q$)
\begin{align} \label{eq:Lambdamod} \widetilde \La_{p,q}^\fS (f_1,f_2) &= \sum_{Q\in \fS}|Q| \jp{f_1}_{Q,p}
\jp{f_2}_{3Q,q},
\\ \label{eq:maxLambdamod}
\La^{**}_{p,q} (f_1,f_2) &=\sup_{\substack{\fS\subset\fD:\\ \fS :\ga\text{-sparse}}}\widetilde\La^\fS_{p,q} (f_1,f_2).
\end{align}

We use, for a cube $S\in \fD$, the notation $\La^{S,**}_{p,q} (f_1,f_2)$ if we require that all the sparse families featuring in the sup consist of cubes contained in $ S$. Recall the definition of $q_*(p,p_\circ, r_\circ) $ in \eqref{eq:r*def}. Fix $p$ and let
\Be\label{eq:Telldef} T_{\ell} = h_{\la(p) ,\ell}(\rho(D)) .\Ee

\begin{definition} Let $d\ge 2$,
$\frac{2(d+1)}{d+3}\le p_\circ<\frac{2d}{d+1}$, $p_\circ \leq r_\circ \leq \frac{d-1}{d+1}p_\circ'$ and assume that $\mathrm{Hyp}(p_\circ, r_\circ)$ holds. Let $1\leq p < p_\circ$ and $q_*(p,p_\circ,r_\circ) < q < p'$.
For $\fn=0,1,2,\dots$ let $\bbU(\fn)\equiv \bbU_{p,q} (\fn)$ be the smallest constant $U$ such that for all bounded measurable functions $f_1$, $f_2$ with compact support and for all $S \in \fD$ with $\nc\le L(S)\le \nc+ \fn$,
\[ \big|\biginn{ \sum_{\ell\le L(S)-\nc} T_{\ell} [f_1 \bbone_{S}] }{f_2\bbone_{3S}}\big| \le U \La^{S,**}_{p,q} (f_1,f_2). \]
\end{definition}

The convolution kernels of $T_\ell$ are Schwartz functions and therefore it is immediate that $\bbU_{p,q}(\fn)$ are finite for all $q\le p'$. Our main task will be to prove for that $\sup_{\fn} \bbU_{p,q} (\fn)<\infty$ for $p$ and $q$ as above.
This will be done by induction, by proving that there is a constant $C$ such that for $\fn\ge 1$
\Be \label{eq:U(n)-iteration} \bbU(\fn)\le \max\{ \bbU( \fn-1), C\}. \Ee

The main iteration step in the sparse domination argument has the same form as in \cite{KeslerLacey}.

\begin{prop} \label{prop:iteration}Let $\frac{2(d+1)}{d+3} \le p_\circ<\frac {2d}{d+1}$, $p_\circ \leq r_\circ \leq \frac{d-1}{d+1}p_\circ'$ and assume that $\hypo$ holds. Let $1\le p< p_\circ$ and $q_*(p,p_\circ,r_\circ)<q < p'$.
Then there is a constant $C>0$ such that for every $S\in\fD_{>0}$ and every bounded $f_1:S\to \bbC$, $f_2: 3S\to \bbC$, there is a collection $\fW$ of disjoint dyadic subcubes of $S$ with the properties
\begin{align}\label{eq:measure}&\Big|\bigcup_{Q\in \fW} Q\Big|\le (1-\ga)\, |S|,
\end{align}
\begin{equation}
\label{eq:recursion} \Big|\sum_{\ell=0}^{L(S)-\nc} \inn{T_\ell f_1}{f_2} \Big| \le C
|S|\jp{f_1}_{S,p} \jp{f_2}_{3S,q} + \sum_{\substack{Q\in \fW_{\ge n_\circ}}} \Big| \sum_{\ell=0} ^{L(Q)-\nc}\inn{T_\ell[f_1\bbone_Q] }{f_2\bbone_{3Q} } \Big|.
\end{equation}

\end{prop}

\begin{proof}[Proof of Theorem \ref{thm:blackbox-sparse}, given Proposition \ref{prop:iteration}] In order to prove \eqref{eq:U(n)-iteration} we fix $\fn\ge 1$ and let $S\in\fD_{\nc+\fn}$. Let $\ep>0$. Let $\fW$ be the family of dyadic subcubes guaranteed by Proposition \ref{prop:iteration} such that \eqref{eq:recursion} holds. Note that $\nc\le L(Q)< n_\circ+ \fn$ for all $Q \in \fW_{\ge\nc}$. Therefore, by the induction hypothesis,
for each $Q\in \fW_{\ge\nc}$ there is a $\gamma$-sparse family $\fS_Q$ of dyadic subcubes of $Q$ such that
\[\Big| \sum_{\ell=0} ^{L(Q)-\nc}\inn{T_\ell[f_1\bbone_Q] }{f_2\bbone_{3Q} } \Big|\le \bbU(\fn-1) \widetilde \La_{p,q}^{\fS_Q} (f_1, f_2) +\ep.
\]
Setting $E_S:=S \backslash \bigcup_{Q \in \fW} Q$, the collection $\fS= \{S\} \cup \bigcup_{Q\in \fW} \fS_Q$ is a $\gamma$-sparse family of dyadic subcubes of $S$ and we have \[ |S|\jp{f_1}_{S,p} \jp{f_2}_{3S,q} +\sum_{Q\in \fW} \widetilde \La_{p,q}^{\fS_Q} (f_1,f_2) \le \widetilde \La_{p,q}^{\fS}(f_1,f_2).\] Since $\ep>0$ was arbitrary we deduce \eqref{eq:U(n)-iteration}.

Finally, if $f_1$, $f_2$ are compactly supported $L^\infty$-functions we choose $N$ so that $2^{\nc+10} \supp(f_1), 2^{\nc+10} \supp(f_2)\subset [-N,N]^d$. By the properties of the Lerner--Nazarov \cite{lerner-nazarov} dyadic lattice $\fD$ and by Lemma \ref{lem:error-kernelest}, there is a cube $S\in \fD$ which contains $[-N,N]^d$ such that $|\inn {T_\ell f_1}{f_2}|= |\inn {T_\ell[f_1\bbone_S]}{f_2\bbone_{3S}} |\lc \ep 2^{-\ell} $ for sufficiently large $\ell$.
Since $\ep>0$ is arbitrary, this together with the main estimate \eqref{eq:U(n)-iteration}, noting from \eqref{hla-decomposition} that $h_\la=\sum_{\ell=0} ^\infty h_{\la,\ell}$, yields the bound
\[|\inn{h_{\la(p)} (\rho(D)) f_1}{f_2}| \lc \La^{**}(f_1,f_2) .\]
A well-known argument relying on the three lattice theorem in \cite{lerner-nazarov} allows to replace $\La^{**}$ by the more standard maximal sparse form $\La^*$ (see e.g. \cite[Ch.4.2]{BRS} for details). Since $\cR_a^{\la(p)}-h_{\la(p)} (\rho(D)) $ satisfies a standard $\Sp(p,p)$ bound for all $p\ge 1$ (see the beginning of \S\ref{sec:decomp}) we obtain the desired $\Sp(p,q)$ bound for $\cR^{\la(p)}_a$.
\end{proof}
\begin{proof}[Proof of Proposition \ref{prop:iteration}] Let $\alpha=\jp{f_1}_{S,p}$ and
$\Omega=\{x: M_{HL}(|f_1|^p)\ge \frac{100 d}{1-\gamma} \alpha^p\}$,
where $M_{HL}$ denotes the Hardy--Littlewood maximal function. Let $\fW$ be the collection of Whitney cubes of $\Omega$ satisfying that $\Omega=\bigcup_{Q \in \fW} Q$ and
\Be\label{eq:Whitney}
\diam(Q)\le \dist(Q,\Omega^\complement)\le 4\diam (Q)
\Ee
for all $Q\in \fW$; see \cite[Ch. VI.1]{Ste70}.
Since $|\Omega| \leq (1-\gamma)|S|$, condition \eqref{eq:measure} follows.

Define next $g=f_1 \bbone_{\Omega^c}$ and $b_Q= f_1 \bbone_Q$ for each $Q \in \fW$. By the standard Calderón--Zygmund properties,
\[
\| g \|_\infty \lesssim \alpha \qquad \text{and} \qquad \int_Q |b_Q|^p \le \alpha^p |Q|.
\]
Let
\[
B_0=\sum_{\substack{Q\in \fW_{\le 0}}} b_Q, \qquad \quad B_j=\sum_{\substack{Q\in \fW_j}} b_Q, \quad \,j>0.\]
With these definitions we have
$\int_Q |B_j|^p\lc \alpha^p |Q|$ whenever $Q$ is a dyadic cube with $L(Q) \geq j \geq 0$; note that we also have $\int_Q|B_j|\le \alpha |Q|. $

Let $\cT^S= \sum_{\ell=0}^{L(S)-\nc} T_\ell $.
Then $|\inn{\cT^S f_1}{f_2}| \le I+II
+ III$, where
\begin{align*}I&=|\inn{\cT^S g}{f_2}|, \qquad
II=\Big|\Biginn{\sum_{\ell=0}^{L(S)-\nc} T_\ell[\sum_{ 0\le j<\ell+\nc} B_j]}{f_2} \Big|,
\end{align*}
\begin{align*}
III&=
\Big|\Biginn{\sum_{\ell=0}^{L(S)-\nc} T_\ell[\sum_{ j\ge \ell+\nc} B_j]}{f_2} \Big|.
\end{align*}

\subsubsection*{Estimation of $I$} $\hypo$ implies $\VV(p,p)$ and together with Lemma \ref{lem:conseq-of-hyp} this implies $\|T_\ell\|_{L^p\to L^p} =O(1)$. Since $\|h_{\la(p),\ell}\|_\infty=O(2^{-\la(p)})$ and $|\frac 1p-\frac 12|>|\frac 1q-\frac 12| $ we get $\|T_\ell\|_{L^q\to L^q}\lc 2^{-\eps\ell} $ for some $\eps>0$, by interpolation.
We can also apply this for the adjoint operators; indeed
$T_\ell^*(h_{\la_*}(\rho(D))^*= h_{\la,*}(\widetilde \rho(D))$ where $\widetilde \rho $ is the Minkowski functional of $-\Omega$. Hence $\|T_\ell^*\|_{L^q\to L^q}\lc 2^{-\eps\ell} $.
Therefore
\begin{align} \notag
I&= |\inn{g \bbone_S}{ (\cT^S)^*f_2} |\lc \alpha \int_S|(\cT^S)^* f_2| \lc
\alpha |S|^{1-1/q} \sum_{\ell=0}^{\infty}\|T_\ell^* f_2\|_{q} \\
\label{Isparsebd}&\lc \alpha |S|^{1-1/q}
\Big(\int_{3S}|f_2|^q\Big)^{1/q} \lc |S| \jp{f_1}_{S,p} \jp{f_2}_{3S,q}.
\end{align}

\subsubsection*{Estimation of $II$}

We estimate $II\le \sum_{s\ge -\nc} II_s$ where (with $s\wedge 0:=\max\{s,0\}$)
\[II_s= \Big|\Biginn{\sum_{\ell=s\wedge 0}^{L(S)-n_\circ } T_\ell B_{\ell-s} }{f_2} \Big|.\]
First assume $s>0$. We use Proposition \ref{CZ-prop} with $r=q'$ and $\fQ=\fW_{\ge 1}$. Then\begin{align*} &\Big\|\sum_{\ell=s+1} ^{L(S)-\nc} T_\ell B_{\ell-s} \Big\|_{q'}^{q'} \lc 2^{-s\ep q'} \alpha^{q'-p} \sum_{Q\in \fW} \|b_Q\|_p^p
\\
&\lc 2^{-s\ep q'} \alpha^{q'-p} \sum_{Q\in \fW} \alpha^p |Q|\lc 2^{-s\eps q'} \alpha^{q'} |S|.
\end{align*}
For the term with $\ell=s$ we use Proposition \ref{CZ-prop} with $r=q'$ and $\fQ=\fD_0$; note that for
$R\in \fD_0$ and $f_R:= \sum_{ Q\in \fW , Q\subset R}b_Q$ we have $ \int_R|f_R|^p \lc \alpha^p|R| $.
In \eqref{eq:CZ-cons} we set $u_\ell=1$ if $\ell=s$ and $u_\ell=0$ for $\ell\neq s$ and obtain
\begin{align*} \big\|T_s B_{0} \big\|_{q'}^{q'} &\lc 2^{-s\ep q'} \alpha^{q'-p} \sum_{R\in \fD_0} \Big\| \sum_{ \substack{ Q\in \fW \\ Q\subset R}}b_Q\Big\|_p^p
\lc 2^{-s\ep q'} \alpha^{q'-p} \sum_{Q\in \fW} \|b_Q\|_p^p
\\& \lc 2^{-s\ep q'} \alpha^{q'-p} \sum_{Q\in \fW} \alpha^p |Q|\lc 2^{-s\eps q'} \alpha^{q'} |S|.
\end{align*}
Finally, the terms with $-\nc\le s\le 0$ are treated by part (ii) of Proposition \ref{CZ-prop} with $n\le \nc$ (so that the polynomial growth of the constant in $n$ is irrelevant). We get
\[ \Big\|\sum_{\ell=0} ^{L(S)-\nc} T_\ell B_{\ell-s} \Big\|_{q'}^{q'} \lc_{\nc} \alpha^{q'} |S|.\]
Combining these estimates we get for all $s\ge -\nc$
\begin{align} \notag
II_s\le \Big\|\sum_{\ell=s\wedge 0}^{L(S)-\nc} T_\ell B_{\ell-s} \Big \|_{q'} \|f_2\|_{q}
&\lc \min\{2^{-s\eps},1\} \alpha |S|^{1-\frac 1q} \Big(\int_{3S}|f_2|^q \Big)^{1/q}
\\& \label{IIs-sparsebd}\lc 2^{-s\eps} |S| \jp{f_1}_{S,p} \jp{f_2}_{3S,q}
\end{align} by the definition $\alpha= \jp{f_1}_{S,p}$. After summing in $s \geq -n_\circ$ we obtain \Be\label{IIsparsebd} II\lc |S| \jp{f_1}_{S,p} \jp{f_2}_{3S,q}. \Ee

\subsubsection*{Estimation of $III$}
Write $III=|\sum_{j\ge \nc} \sum_{\ell=0}^{j-\nc} \biginn{ T_\ell[\sum_{Q\in \fW_j} b_{Q}]}{f_2}|$
and estimate $III\le III_{\mathrm{main}} + III_{\mathrm{err}}$
where
\Be\label{III-main-and err}\begin{aligned}
&III_{\mathrm{main}} = \sum_{\substack{Q\in \fW}} \Big|\biginn{ \sum_{\ell=0}^{L(Q)-n_\circ} T_\ell [f_1\bbone_Q] }{f_2\bbone_{3Q} }\Big|,
\\
& III_{\mathrm{err} }= \sum_{\substack{Q\in \fW} }\Big|\biginn{ \sum_{\ell=0}^{L(Q)-n_\circ} T_\ell [f_1\bbone_Q] }{f_2\bbone_{(3Q)^\complement} }\Big|.
\end{aligned}
\Ee
Note that $III_{\mathrm{main}}$ is the last term in \eqref{eq:recursion}.
By the estimations for $I$ and $II$
we are done if we prove the stronger estimate
\Be\label{IVsparsebd}
III_{\mathrm{err}}\lc |S|\jp {f_1}_{S,p} \jp{f_2}_{3S,1}
\Ee since by H\"older's inequality $\jp {f_2}_{3S,1} \lesssim \jp{f_2}_{3S,q}$.

To see \eqref{IVsparsebd} we use Lemma \ref{lem:error-kernelest} and estimate $III_{\mathrm{err}}$ by
\begin{align*}
&\sum_{\substack{Q\in\fW}} \int_Q|f_1(y) |
\int_{ (3Q)^\complement\cap 3S}
|x-y|^{-N}|f_2(x)|\ud x
\ud y
\\ &\lc \sum_{m=n_\circ}^\infty\sum_{ \substack{Q\in\fW_m} } \alpha |Q| \sum_{n=1}^\infty
2^{-(n+m) N}
\int_{ 2^{n+1} Q}
|f_2|
\\&\lc \alpha\sum_{m,n} 2^{-(m+n) (N-d) } \sum_{R\in \fD_{m+n} }
\int_{3R}|f_2| \lc \alpha \int_{3S} |f_2|
\end{align*}
which gives \eqref{IVsparsebd}.
\end{proof}

\section{Auxiliary estimates for the proof of Theorem \ref{thm:STendpt}}\label{sec:auxiliary-for Thm1.3}
We consider multiplier transformations acting on families of functions $F=\{f_Q\}$, with $f_Q:\bbR^d\to \bbC$ in $L^p$, indexed by cubes $Q\in \fD_{\ge 0}$. These functions are assumed to belong to weighted $\ell^r(L^p)$ spaces $\Xspace_{p,r}$ of vector-valued functions, with norm
\Be \label{Xp2}
\|F\|_{\Xspace_{p,r}}= \Big (\sum_{j=0}^\infty\sum_{Q\in \fD_j} [2^{-j d(\frac 1p-\frac 1r)} \|f_Q\|_p]^r\Big)^{1/r}.
\Ee
In the present paper we take $r=2$.
The following result is equivalent with $\VV(\frac{2(d+1)}{d+3},2)$. Fix $\la_*= \frac{d-1}{2(d+1)}$ as in Theorem \ref{thm:STendpt} and let

\Be \label{eq:new-building-blocks*} {T_{\la_*,\ell} f} = h_{\la_*,\ell} (\rho(D) ) f.\Ee
\begin{lemma} \label{lem:endpoint-ST-summation}
Let $p= \frac{2(d+1)}{d+3}$.
Then for $0\le \nu\le 4\nc$,
\Be\label{eq:improved-ST} \Big\| \sum_{\ell=0}^\infty u_\ell T_{\la_*,\ell} \big[\sum_{j=0}^{\ell+\nu} \sum_{Q\in \fD_j} f_Q] \Big\|_2 \lc \|u\|_\infty \|F\|_{\Xspace_{p,2}}
\Ee
for any bounded sequence $u=\{u_\ell\}_{\ell=0}^\infty$. Moreover,
\Be\label{eq:improved-ST-sq}
\Big(\sum_{\ell=0}^\infty
\Big\|T_{\la_*,\ell} \big[\sum_{j=0}^{\ell+\nu} \sum_{Q\in \fD_j} f_Q\big] \Big\|_2^2\Big)^{1/2}\lesssim \| F\|_{\Xspace_{p,2}}.
\Ee
\end{lemma}
\begin{proof} The inequality \eqref{eq:improved-ST-sq} is a formal consequence of
\eqref{eq:improved-ST}; this can be seen by taking $u_\ell=r_\ell(t)$ where the $r_\ell$ are the Rademacher functions, and then averaging in $t$.
Following an idea in \cite{tao-STendpoint} we may split, for any choice of $j$,
\[ {T_{\la_*,\ell} } = 2^{-j\la_*} \om_{\ell,j} (\rho(D)) \vth_j(\rho(D)), \] where
\[\om_{\ell,j} (\varrho)= 2^{(j-\ell) \la_*} \frac{2^{\ell \la_*} h_{\la_*,\ell}(\varrho)} {\vth_j(\varrho)}, \quad \text{ with $\vth_j(\varrho)= (1+2^{2j} (1-\varrho)^2)^{-d}$.} \]
We change variables $\ell=j+n$.
Using the estimates in Lemma \ref{lem:ptwisebdhellla} with large $N_1$ we obtain for each $n\ge -\nu$ and for $\varrho=\rho(\xi)\in \supp(\chi)$,
\begin{align*} |\om_{j+n,j}(\varrho) | &\lc 2^{-n\la_* }
\frac{ \min\{ (2^{(j+n) } |1-\varrho|)^{N_\circ+1}, (1+2^{j+n}|1-\varrho|)^{-N_1} \} } { (1+2^{2j}|1-\varrho|^2)^{-d} }
\end{align*} and from this
\[ \sup_\xi \sum_{j\ge 0 \wedge -n} |\om_{j+n,j}(\rho(\xi))| \le C 2^{-n\la_*}\] with $C$ independent of $n$ (only dependent of $|\nu|\lc \nc$ which is fixed).
Hence with $g_j:=\sum_{Q\in \fD_j} f_Q$
\begin{align}
\notag\Big\|\sum_{\ell=0}^\infty u_\ell T_{\la_*,\ell} \sum_{j=0}^{\ell+\nc} g_j \Big\|_2
&= \sum_{n =-\nc}^\infty \Big\| \sum_{j\ge 0\wedge -n} 2^{-j\la_*} u_{j+n}\om_{j+n,j}(\rho(D)) \vth_j(\rho(D)) g_j \Big\|_2
\\
\label{eq:orth-in-j}
&\lc \sum_{n =-\nc}^{\infty } 2^{-n\la_*} \Big(\sum_{j\ge 0\wedge -n} |u_{j+n} |^2\big\|
2^{-j\la_*} \vth_j(\rho(D)) g_j\big\|_2^2\Big)^{1/2}.
\end{align}
For $j\ge 0$ we have by \eqref{eq:thetaj-ST-est}
\[\big\|
2^{-j\la_*} \vth_j(\rho(D)) g_j\|_2 \lc \Big( \sum_{Q\in \fD_j} 2^{-jd(\frac 2p-1)} \|f_Q\|_p^2 \Big)^{\frac 12} \]
and \eqref{eq:improved-ST} follows by combining the two previous displays.
\end{proof}

In order to prove Theorem \ref{thm:STendpt} for $(\frac 1p,\frac 1q) $ on the open edge connecting $(\frac{d+3}{2(d+1)},\frac 12) $ with
$(\frac 12, \frac{d+3}{2(d+1)}) $
we need a refined bilinear variant of Lemma \ref{lem:endpoint-ST-summation}.

Let $u\in\ell^\infty(\bbN_0)$, $F_i
=\{f_{i,Q}\}_{Q\in \fD_{\ge 0}}$, $i=1,2$, $\ell, j_1,j_2\ge 0$, $\la_*=\frac{d-1}{2(d+1)}$ and $\fQ, \fQ'\subset \fD_{\ge 0}$. Define
\Be\label{eq:Gammajsdef}
\Gamma^{\ell,j_1,j_2}_{\fQ,\fQ'}(F_1,F_2):= \Biginn{
T_{\la_*,\ell}\big[ \sum_{Q\in \fQ_{j_1} } f_{1,Q }\bbone_Q \big] }{\sum_{Q'\in\fQ_{j_2}' } f_{2,Q'} \bbone_{Q'} }.
\Ee
Next define a family of bilinear forms, depending on parameters $0\le \nu_1, \nu_2\le 4\nc$ by
\Be\label{varGammadefinition}
\varGamma_{\fQ,\fQ'} (F_1,F_2,u) := \sum_{\ell\ge 0} u_\ell \sum_{0\le j_1< \ell+\nu_1}\sum_{0\le j_2\le \ell+\nu_2} \Gamma^{\ell, j_1,j_2}_{\fQ,\fQ'}(F_1,F_2).\Ee

\begin{prop}\label{prop:bilinear-var}
Let
$\fQ\subset \fD_{\ge 0}$, $\fQ'\subset \fD_{\ge 0}$ each be disjoint families of cubes such that
\Be \label{eq:separation}
\dist(Q,Q') >\tfrac 12 \diam (Q')
\Ee for all $(Q,Q') \in \fQ\times \fQ'$ satisfying $L(Q')\ge L(Q)+4$.
Then for $(\frac1p,\frac1q)$ on the closed edge connecting $(\frac{d+3}{2(d+1)},\frac 12) $ with
$(\frac 12, \frac{d+3}{2(d+1)}) $,
\Be \label{eq:bilinear} \big| \varGamma_{\fQ,\fQ'} (F_1,F_2,u) \big| \lc\|u\|_\infty \|F_1\|_{\Xspace_{p,2} } \|F_2\|_{\Xspace_{q,2}}.
\Ee
\end{prop}

We need an auxiliary lemma that states that under the separation condition \eqref{eq:separation} (which is common in Whitney type decompositions) the terms $\Gamma_{\ell,j_1,j_2} $ are negligible when $j_1\le \ell\ll j_2$. Here we have essentially no restrictions on $p,q,r$.

\begin{lemma} \label{lem:bilinear-error}
Let
$\fQ\subset \fD_{\ge 0}$, $\fQ'\subset \fD_{\ge 0}$ for which the separation condition \eqref{eq:separation} is satisfied.
Let $1\le r<\infty$, $1\le p,q\le\infty$.
Then
for any $N \geq 0$, $0\le j_1\le \ell+2\nc$, $j_2> \ell+3\nc$,
\Be\big |\Gamma_{\fQ,\fQ'}^{\ell, j_1,j_2}(F_1,F_2)\big|\lc 2^{-j_2 N} \|F_1\|_{\Xspace_{p,r}} \|F_2\|_{\Xspace_{q,r'}} \,.\Ee
\end{lemma}
\begin{proof}
Note that if $j_1\le \ell+2\nc$, $j_2>\ell+3\nc$, $Q\in \fQ_{j_1}$, $Q'\in \fQ_{j_2}'$ then $j_2\ge 4+j_1$ and thus \eqref{eq:separation}
holds by assumption. This separation condition implies
$\dist( 2^{\ell+\nc} Q, Q') \ge \sqrt {d} (2^{j_2-1}- 2^{\ell+\nc+1})\ge 2^{j_2-2} \ge 2^{\ell+2\nc}$ so that the first
estimate in Lemma \ref{lem:error-kernelest} applies.
That is, if
$K_{\la^*,\ell}=\cF^{-1}[h_{\la_*,\ell} \circ\rho]$ and $x\in Q'$, $y\in Q$ then $|K_{\la^*,\ell}(x-y)| \lc_{N_1} |x-y|^{-N_1}$, and $|x-y| \approx \dist(Q,Q')$. This yields for fixed $j_1,j_2$ the bound
\begin{multline} \label{eq:firstGammabd} |\Gamma_{\fQ,\fQ'}^{\ell, j_1,j_2} (F_1,F_2)| \lc \sum_{m=0}^\infty 2^{-(j_2+m) N_1} \sum_{Q'\in \fQ_{j_2}'} \cI^{m}_{j_1,j_2,Q'},\\\text{with } \cI^{m}_{j_1,j_2,Q'} = \int_{Q'} |f_{2,Q'}(x)|
\sum_{ \substack{Q\in \fQ_{j_1}: 2^{j_2+m-2} \le\\ \dist (Q,Q') < 2^{j_2+m+2}}}
\int_Q |f_{1,Q}(y)|\ud y \ud x.
\end{multline}
Now,
\begin{align*}
&\sum_{Q' \in \fQ_{j_2}'}\cI^{m}_{j_1,j_2,Q'}\lc 2^{ j_2\frac{d}{q'} } \sum_{Q'\in \fQ_{j_2}'} \|f_{2,Q'} \|_q
\sum_{\substack{Q\in \fQ_{j_1}: \dist (Q,Q') \approx2^{j_2+m}}}
2^{j_1\frac d{p'}} \|f_{1,Q}\|_{p}
\\
&\lc 2^{j_2\frac{d}{q'} } \sum_{Q'\in\fQ_{j_2}'} \|f_{2,Q'} \|_q
\Big(\sum_{ Q\in \fQ_{j_1}: \dist (Q,Q') \approx 2^{j_2+m}}
2^{j_1\frac{rd}{p'}} \|f_{1,Q}\|_{p}^{r}\Big)^{1/r} 2^{ (j_2+m-j_1) \frac d{r'}}
\\
&
\lesssim 2^{m\frac d{r'}} 2^{j_2d( 2-\frac 1r-\frac 1q) +j_1d(\frac 1r-\frac 1p)} \Big(\sum_{Q'\in \fQ_{j_2} '} \|f_{2,Q'}\|_q^{r'}\Big)^{\frac 1{r'}} \Big(
\sum_{\substack {Q\in \fQ_{j_1} , Q'\in \fQ_{j_2}' \\ \dist (Q,Q') \approx 2^{j_2+m}} }
\|f_{1,Q}\|_{p}^r\Big)^{\frac 1r}
\\&\lc 2^{(j_2+m)d }
\Big(\sum_{Q\in \fQ_{j_1}} \big[2^{-j_1 d(\frac 1{p} -\frac 1r)} \|f_{1,Q}\|_{p} \big]^r \Big)^{\frac 1r}
\Big(\sum_{Q\in \fQ_{j_2}'} \big[2^{-j_2 d(\frac 1{q} -\frac 1{r'})} \|f_{2,Q}\|_{q} \big]^{r'} \Big)^{\frac 1{r'}}
\end{align*}
where in the last inequality we have used that for any $Q \in \fQ_{j_1}$ we have \[ \#\{Q' \in \fQ_{j_2}': \dist(Q,Q') \approx 2^{j_2+m}\} \approx 2^{md}.\]
The claimed bound now follows immediately from
\eqref{eq:firstGammabd} with $N_1>N+d$.
\end{proof}

\begin{proof}[Proof of Proposition \ref{prop:bilinear-var}]
First consider the case $p=p_1=\frac{2(d+1) }{d+3}$;
$q=q_1=2$.
We let $G^\ell=\sum_{0\le j_2\le \ell+\nu_2} \sum_{Q'\in\fQ_{j_2}'} f_{2,Q'} \bbone_{Q'}$, and observe that
\Be\label{eq:Gell-expr}\varGamma_{\fQ,\fQ'} (F_1,F_2,u) =
\sum_{\ell\ge 0} u_\ell \Biginn{T_{\la_*,\ell} \big[ \sum_{\substack{0\le j_1< \ell+\nu_1,\\Q\in \fQ_{j_1}}} f_{1,Q }\bbone_Q\big] }{G^\ell} .
\Ee
Split
\Be \label{eq:GME} G^\ell=G-M_\ell- E_\ell, \quad \text{ with} \quad
G=\sum_{Q'\in\fQ_{\ge 0}' } f_{2,Q'} \bbone_{Q'},
\Ee\[
M_\ell=\sum_{\substack{j_2=\ell+\nu_2+1}}^{\ell+4\nc} \sum_{Q'\in\fQ_{j_2}' } f_{2,Q'} \bbone_{Q'} , \quad
E_\ell=\sum_{Q'\in\fQ_{>\ell+4\nc}' } f_{2,Q'} \bbone_{Q'}.
\] We then have
\[
\varGamma_{\fQ,\fQ'} =
\varGamma^{\mathrm{full}}_{\fQ, \fQ'} -
\varGamma^{\mathrm{mid}}_{\fQ, \fQ'} -
\varGamma^{\mathrm{err}}_{\fQ, \fQ'}
\]
where
$\varGamma^{\mathrm{full}}_{\fQ, \fQ'} (F_1,F_2,u)$, $
\varGamma^{\mathrm{mid}}_{\fQ, \fQ'} (F_1,F_2,u) $ and $
\varGamma^{\mathrm{err}}_{\fQ, \fQ'} (F_1,F_2,u) $
are defined as in \eqref{eq:Gell-expr} but with $G^\ell$ replaced by $G$, $M_\ell$ and $E_\ell$, respectively.
By Lemma \ref{lem:endpoint-ST-summation} and the Cauchy-Schwarz inequality we see that
\begin{align}\notag \varGamma^{\mathrm{full}}_{\fQ, \fQ'} (F_1,F_2,u)
&=\Big| \sum_{\ell=0}^\infty u_\ell \Biginn{
T_{\la_*,\ell}
\big[\sum_{\substack{0\le j_1< \ell+\nu_1,\\Q\in \fQ_{j_1}}} f_{1,Q }\bbone_Q \big] }{ G} \Big|
\\& \notag
\lc \|u\|_\infty
\|F_1\|_{\Xspace_{p_1,2}}
\|G\|_2\lc \|u\|_\infty
\|F_1\|_{\Xspace_{p_1,2}} \|F_2\|_{\Xspace_{2,2}}
\end{align}
where the last inequality uses the disjointness of the cubes in $\fQ'$.
Next we
apply the Cauchy-Schwarz inequality with respect to $x$ and $\ell$ and get
\begin{align} \notag \big|&\varGamma^{\mathrm{mid}}_{\fQ, \fQ'} (F_1,F_2,u)\big|=
\Big| \sum_{\ell=0}^\infty u_\ell \Biginn{T_{\la_*,\ell}
\big[\sum_{\substack{0\le j_1< \ell+\nu_1,\\Q\in \fQ_{j_1}}} f_{1,Q }\bbone_Q \big] }{ M_\ell} \Big|
\\ \notag &
\lc\Big(\sum_{\ell=0}^\infty|u_\ell|^2 \Big\|T_{\la_*,\ell}
\big[\sum_{\substack{0\le j_1< \ell+\nu_1,\\Q\in \fQ_{j_1}}} f_{1,Q }\bbone_Q] \Big\|_2^2\Big)^{1/2} \Big(\sum_{\ell=0}^\infty \|M_\ell\|_2^2\Big)^{1/2}
\\ \notag
& \lc \|u\|_\infty \|F_1\|_{\Xspace_{p_1,2}} \|F_2\|_{\Xspace_{2,2}}
\end{align}
where in the last inequality we have used Lemma \ref{lem:endpoint-ST-summation}
and the square-function estimate $(\sum_{\ell=0}^\infty \|M_\ell\|_2^2)^{1/2} \lc \|F_2\|_{\Xspace_{2,2}}$.
For the terms involving the $E_\ell$ we use Lemma \ref{lem:bilinear-error} and obtain
\begin{align} \notag
&\big|\varGamma^{\mathrm{err}}_{\fQ, \fQ'} (F_1,F_2,u)\big|=\Big| \sum_{\ell=0}^\infty u_\ell \Biginn{T_{\la_*,\ell}
\big[\sum_{\substack{0<j_1< \ell+\nu_1,\\Q\in \fQ_{j_1}}} f_{1,Q }\bbone_Q \big] }{ E_\ell} \Big|
\\ \notag
&\lc\sum_{\ell=0}^{\infty} \sum_{0 < j_1< \ell+\nu_1}\sum_{j_2\ge \ell+4\nc} 2^{-j_2N} \|u\|_\infty\|F_1\|_{\Xspace_{p_1,2}}\|F_2\|_{\Xspace_{2,2}} \lc \|u\|_\infty\|F_1\|_{\Xspace_{p_1,2}}\|F_2\|_{\Xspace_{2,2}}.
\end{align}

Combining the estimates for $\varGamma^{\mathrm{full}}_{\fQ, \fQ'} $, $
\varGamma^{\mathrm{mid}}_{\fQ, \fQ'} $ and $
\varGamma^{\mathrm{err}}_{\fQ, \fQ'} $
we get
\Be\label{eq:oneendpt} \big| \varGamma_{\fQ,\fQ'} (F_1,F_2,u) \big|\lc \|u\|_\infty \|F_1\|_{\Xspace_{p_1,2} }\|F_2\|_{\Xspace_{2,2}}.\Ee
Next observe that
\[\varGamma_{\fQ,\fQ'}(F_1,F_2,u) = \widetilde \varGamma_{\fQ', \fQ} (F_2,F_1,u), \] where $\widetilde \varGamma$ is the bilinear form associated with the domain $-\Om$.
Hence we get
\Be \label{eq:otherendpt} \big|\varGamma_{\fQ,\fQ'} (F_1,F_2,u)\big|\lc \|u\|_\infty \|F_1\|_{\Xspace_{2,2} }\|F_2\|_{\Xspace_{p_1,2}}.\Ee
It is straightforward to show the interpolation formula $[\Xspace_{q_1,2},\Xspace_{q_2,2}]_\theta =\Xspace_{q,2} $ for the Calder\'on complex interpolation spaces with $(1-\theta)/q_1+\theta/q_2=1/q$, $1\le q_1,q_2<\infty$.
Thus the assertion \eqref{eq:bilinear} follows by complex interpolation of
\eqref{eq:oneendpt} and \eqref{eq:otherendpt}.
\end{proof}

\section{Sparse domination, Part II} \label{sec:sparse-partII}
We now prove Theorem \ref{thm:STendpt}. We have $\la_*=\frac{d-1}{2(d+1)}$ and it only remains to prove the $\Sp(p,q)$ bound for $(\frac 1p,\frac 1q)$ on the closed line segment connecting $(\frac{d+3}{2(d+1)}, \frac 12)$ and
$(\frac 12, \frac{d+3}{2(d+1)})$; note that the points on this segment satisfy $\frac 1p+\frac 1q=\frac{d+2}{d+1}$.
As the point $(\frac{d+2}{2(d+1)},\frac{d+2}{2(d+1)}) $ is the center of this line segment we may, by symmetry of sparse bounds, assume that $\frac 1p\ge \frac {d+2}{2(d+1)} $. As before $T_{\la_*,\ell}=h_{\la_*,\ell} (\rho(D) ) $ as in \eqref{eq:new-building-blocks*}.

Setting up an induction argument as in \S\ref{sec:sparse-partI} one reduces the proof of the sparse bound to the following proposition which contains the main iteration step.
\begin{prop}
Let $\frac{2(d+1)}{d+3} \le p\le \frac{2(d+1)}{d+2}$, $\frac 1q=\frac{d+2}{d+1}-\frac 1p$.
Then there is a constant $C>0$ such that for every
$S\in\fD_{>0}$ and bounded $f_1:S\to \bbC$, $f_2: 3S\to \bbC$, there is a collection $\fW$ of disjoint dyadic subcubes of $S$ with the properties
\begin{align}
\label{eq:measure-2-new}&\Big|\bigcup_{Q\in \fW} Q\Big|\le (1-\ga)\, |S|,
\end{align}
\begin{multline}\label{eq:recursion2} \Big|\sum_{\ell=0}^{L(S)-\nc} \inn{ T_{\la_*,\ell} f_1}{f_2} \Big| \le C
|S|\jp{f_1}_{S,p} \jp{f_2}_{3S,q} +\sum_{Q\in \fW_{\geq n_\circ}} \Big| \sum_{\ell=0} ^{L(Q)-\nc}\inn{T_{\la_*,\ell}[f_1\bbone_Q] }{f_2\bbone_{3Q} } \Big|
.\end{multline}

\end{prop}
\begin{proof} Let $\alpha_1=\jp{f_1}_{S,p}$, $\alpha_2=\jp{f_2}_{3S,q} $ and let $\Omega=\Omega_1\cup\Omega_2$ where
\[ \Omega_1=\{x: M_{HL}(|f_1|^p)\ge \tfrac{100 d}{1-\gamma} \alpha_1^p\}, \quad \Omega_2=
\{x\in 3S: M_{HL}(|f_2|^q)\ge \tfrac{100 d}{1-\gamma} \alpha_2^q\} .\]
Let $\fW$ consist of the subcubes of $S$ which are Whitney cubes of $\Omega$. Since $|\Omega| \leq (1-\gamma)|S|$, \eqref{eq:measure-2-new} immediately follows. Observe that the pair of collections $(\fW,\fW)$ satisfies the separation condition \eqref{eq:separation}. Indeed, let $Q, Q'\in \fW$ such that $L(Q')\ge L(Q)+4$, i.e. $\diam(Q')\ge 16\diam (Q)$. There is $x\in \Omega^\complement$ such that $\dist(x,Q)\le 4\diam(Q) \le \frac 14 \diam(Q')$ and therefore
\[
\dist(Q,Q')\ge \dist(Q',x)-4\diam(Q) \ge \diam(Q')-4\diam(Q) \geq \tfrac{3}{4}\diam(Q'),
\]
from which \eqref{eq:separation} holds.
Below we will also use that if $\fQ_0$ is the collection of $Q\in \fD_0$ such that $Q$ contains a cube in $\fW$ then the pair $(\fQ,\fQ'):=(\fQ_0, \fW_{>0})$ also satisfies \eqref{eq:separation}. This is shown by a similar argument. Namely if $Q'\in \fW_{>0}$, with $L(Q')\ge 4$ and $Q\in \fQ_0$, $\widetilde Q\in \fW$ with $\widetilde Q\subseteq Q$ then by the above argument
$\dist (\widetilde Q, Q') \ge \frac 34\diam(Q')$ and thus
$\dist (Q,Q') \ge \frac 34 \diam(Q')-\diam (Q)\ge (\frac 34-\frac 1{16} ) \diam (Q').$

Define $g_i=f_i\bbone_{\Omega^\complement}$ and $b_{i,Q}=f_i\bbone_Q$, for $i=1,2$.
Then \[ \| g_i \|_\infty \lesssim \alpha_i, \qquad \int_Q |b_{1,Q} |^p \le \alpha_1^p |Q|, \qquad
\int_Q |b_{2,Q} |^q \le \alpha_2^q |Q|.\]
For $j>0$, $i=1,2$, let
\[
B_{i,0}= \sum_{\substack{Q \in \fW_{\le 0}}} b_{i,Q}, \qquad B_{i,j}= \sum_{\substack{Q \in \fW_j}} b_{i,Q}.
\]

Setting again $\cT^S= \sum_{\ell=0}^{L(S)-\nc} T_{\la_*,\ell} $ we have
\[ |\inn{\cT^S f_1}{f_2}| \le I+II+III\]
where \[ I=\big|\inn{\cT^S g_1}{f_2}\big|, \quad II= \Big|\Biginn{\sum_{\ell=0}^{L(S)-\nc} T_{\la_*,\ell}[\sum_{ j=0}^{\ell+\nc} B_{1,j}]}{f_2} \Big|\]
and
\[III= \Big|\Biginn{\sum_{\ell=0}^{L(S)-\nc} T_{\la_*,\ell}[\sum_{ j>\ell+\nc} B_{1,j}]}{ f_2} \Big|.\]

Below it will be advantageous to also use the definitions, for $Q\in \fD_{\ge 0}$ and $i=1,2$,
\Be\label{eq:defofBoneupperQ} B_i^Q= \begin{cases}
b_{i,Q} &\text{ if $Q\in \fW_{>0}$,}
\\
0 &\text{ if $Q\notin \fW$, $L(Q)>0$,}
\\ \sum_{\substack { Q'\in \fW\\ Q'\subset Q} }b_{i,Q'} &\text{ if $L(Q)=0$} \,.
\end{cases} \Ee
With these definitions we have $\int|B_i^Q(x)|^p \ud x\lc \alpha_i^p |Q|$, for $i=1,2$ and any $Q\in\fD_{\ge 0}$.

\subsection{\texorpdfstring{Estimation of the terms $I$ and $III$}{Estimation of the terms I and III}} We get $|I|\lc |S|\alpha_1\alpha_2$ by exactly the same argument as used in \eqref{Isparsebd} (replacing $\alpha$ there by $\alpha_1$).

Regarding $III$,
the estimation is identical to the estimation of $III$ in the proof of Proposition \ref{prop:iteration}.
We bound $III\le III_{\mathrm{main}} + III_{\mathrm{err}}$
with the definition of these terms in \eqref{III-main-and err}; the
main term matches the second term on the right-hand side of \eqref{eq:recursion2} and the error term is as before estimated by $|S|\jp{f_1}_{S,p}\jp{f_2}_{3S,q}$ using Lemma \ref{lem:error-kernelest}.

\subsection{\texorpdfstring{Estimation of $II$, in the case $p=\frac{2(d+1)}{d+3}$}{Estimation of II at the Stein-Tomas endpoint}}
This is very similar to the bound for the term $II$ in the proof of Proposition \ref{prop:iteration}, except that now we use the improved bound of
Lemma \ref{lem:endpoint-ST-summation} for $q=2$.
By Lemma \ref{lem:endpoint-ST-summation},
\begin{align*}
&\Big\|\sum_{\ell=0}^{L(S)-\nc} T_{\la_*,\ell} \sum_{j=1}^{\ell+\nc} B_{1,j}\Big\|_2 \lc
\Big(\sum_{j\ge 1} \sum_{Q\in \fW_j} 2^{-2j d(\frac 1p-\frac 12) } \|B_{1,j}\|_p^2 \Big)^{\frac 12}
\\
&\lc \alpha_1^{1-\frac p2} \Big(\sum_{Q\in \fW_{>0}} \|B_1^Q\|_p^p \Big)^{\frac 12}
\lc \alpha_1^{1-\frac p2} \Big(\alpha_1^p \sum_{Q\in \fW_{>0}} |Q|\Big)^{\frac 12}
\lc \alpha_1 |S|^{1/2} \lc |S|^{1/2} \jp{f_1}_{S,p}.
\end{align*}
Moreover,
\begin{align*}
&\Big\|\sum_{\ell=0}^{L(S)-\nc} T_{\la_*,\ell} B_{1,0} \Big\|_2 \lc
\Big(\sum_{Q\in \fD_0} \|B_1^Q\|_p^2 \Big)^{\frac 12}
\lc \alpha_1^{1-\frac p2} \Big(\sum_{Q\in \fD_0} \|B_1^Q\|_p^p \Big)^{\frac 12}
\\& \lc \alpha_1^{1-\frac p2} \Big(\alpha_1^p \sum_{Q'\in \fD_0} \sum_{\substack {Q\in \fW\\Q\subset Q'} }|Q|\Big)^{\frac 12}
\lc \alpha_1 |S|^{1/2} \lc |S|^{1/2} \jp{f_1}_{S,p}.
\end{align*}
Combining these two estimates and applying the Cauchy-Schwarz inequality we obtain
\[ II \le \Big\|\sum_{\ell=0}^{L(S)-\nc} T_{\la_*,\ell} \sum_{j=0}^{\ell+\nc} B_{1,j} \Big\|_2 \Big(\int_{3S} |f_2|^2\Big)^{1/2} \lc |S| \jp{f_1}_{S,p} \jp{f_2}_{3S,2} .\]

\subsection{\texorpdfstring{Estimation of $II$, in the case $\frac{2(d+1)}{d+3} < p\le \frac{2(d+1)}{d+2} $}{Estimation of II for intermediate exponents}}
We now split $II$ further as $II\leq II_1+II_2+II_3$ where
\begin{align*}
II_1=\Big|\Biginn{\sum_{\ell=0}^{L(S)-\nc} T_{\la_*,\ell}[\sum_{ j=0}^{\ell+\nc} & B_{1,j}]}{g_2}\Big|,
\quad
II_2= \Big|\sum_{\ell=0}^{L(S)-\nc} \Biginn{ T_{\la_*,\ell}[\sum_{ j=0}^{\ell+\nc} B_{1,j}]}{ \sum_{ j'=0}^{ \ell+4\nc} B_{2,j'}}\Big|,
\end{align*}
\begin{align*}
II_3=\Big| \sum_{\ell=0}^{L(S)-\nc} \Biginn{T_{\la_*,\ell}[\sum_{ j=0}^{\ell+\nc} B_{1,j}]}{ \sum_{ j'> \ell+4\nc} B_{2,j'}}\Big|.
\end{align*}
\subsubsection{Estimation of $II_1$}
Since now in the given range we have $\|T_{\la_*,\ell}\|_{L^p\to L^p} \lc 2^{-\ell\eps(p)}$ we obtain by H\"older's inequality
\begin{align*} II_1 &\le \sum_{\ell=0} ^{L(S) -\nc} 2^{-\ell\eps(p)} \Big\|\sum_{j=0}^{\ell+\nc} B_{1,j} \Big\|_p \|g_2\|_{p'}
\lc \Big(\int_S|f_1|^p\Big)^{1/p} \Big( \int_{3S} |g_2|^{p'}\Big)^{1/p'}
\\&\lc |S|^{1/p} \jp{f_1}_{S,p} |S|^{1/p'} \alpha_2 \lc |S|\jp{f_1}_{S,p} \jp{f_2}_{3S,q}.
\end{align*}

\subsubsection{Estimation of $II_2$} Now $\frac 1q=\frac{d+2}{d+1}-\frac 1p$.
We split $II_2\leq \sum_{i=1}^4 II_{2,i} $ where
\begin{align*} II_{2,1} &= \Big| \sum_{\ell=0}^{L(S)-\nc} \Biginn{ T_{\la_*,\ell}[\sum_{ j=1}^{\ell+\nc} B_{1,j}]}{ \sum_{ j'=1}^{ \ell+4\nc} B_{2,j'}}\Big|,
\,\, II_{2,2} = \Big|\sum_{\ell=0}^{L(S)-\nc} \Biginn{ T_{\la_*,\ell}[B_{1,0}]}{ \sum_{ j'=1}^{ \ell+4\nc} B_{2,j'}}\Big|,
\\ II_{2,3} &=\Big| \sum_{\ell=0}^{L(S)-\nc} \Biginn{ T_{\la_*,\ell}[\sum_{ j=1}^{\ell+\nc} B_{1,j}]}{ B_{2,0}}\Big|,
\quad II_{2,4} = \Big|\sum_{\ell=0}^{L(S)-\nc} \biginn{ T_{\la_*,\ell}[ B_{1,0}]}{ B_{2,0}}\Big|.
\end{align*}

We first consider $II_{2,1}$ and
apply Proposition \ref{prop:bilinear-var},
with $\nu_1=\nc$ and $\nu_2=4\nc$,
letting $\fQ=\fQ'$ be the family of all cubes in $\fW_{>0}$ so that the separation condition
\eqref{eq:separation} is satisfied.
We then obtain
\[
II_{2,1} \lc \Big (\sum_{Q\in \fW} 2^{-2L(Q) d(\frac 1p-\frac 12)} \|b_{1,Q}\|_p^2\Big)^{1/2} \Big (\sum_{Q'\in \fW} 2^{-2L(Q') d(\frac 1q-\frac 12)} \|b_{2,Q'}\|_q^2\Big)^{\frac 12} \] and write the right-hand side as $II_{2,1}(p)II_{2,1}(q)$. We have
\begin{align*}
&II_{2,1} (p)\lc \Big (\sum_{Q\in \fW} 2^{-2L(Q) d(\frac 1p-\frac 12)} \|b_{1,Q}\|_p^2\Big)^{1/2} \\&
\lc \Big (\sum_{Q\in \fW} 2^{-2L(Q) d(\frac 1p-\frac 12)} (\alpha_1^p|Q|)^{2/p} \Big)^{1/2} \lc \Big (\sum_{Q\in \fW} |Q| \Big)^{1/2} \alpha_1\lc |S|^{1/2}\alpha_1.
\end{align*} In exactly the same way we obtain $II_{2,1} (q)\lc |S|^{1/2}\alpha_2$
and hence $II_{2,1} \lc|S|\alpha_1\alpha_2.$

The expressions $II_{2,2}$, $II_{2,3}$ and $II_{2,4}$ are bounded similarly.
For $II_{2,2}$ we let $\fQ_0$ be the family of dyadic unit cubes $Q$ with the property that $Q$ contains a cube in $\fW$, and $\fQ'=\fW_{>0}$.
As observed in our discussion at the beginning of the proof we have the separation condition \eqref{eq:separation} in this case.
Applying Proposition
\ref{prop:bilinear-var} to $\varGamma_{\fQ_0,\fW_{>0}}$
we get
\[
II_{2,2} \lc \Big (\sum_{Q\in\fD_0} \Big\|\sum_{\substack{W\in \fW\\W\subset Q} }b_{1,W} \Big\|_p^2\Big)^{\frac 12} \Big (\sum_{Q'\in \fW} 2^{-2L(Q') d(\frac 1q-\frac 12)} \|b_{2,Q'}\|_q^2\Big)^{\frac 12} \]
which we write as $\widetilde{II}_{2,2}(p)II_{2,2}(q)$. Note that $II_{2,2}(q) = II_{2,1}(q) \lc |S|^{1/2} \alpha_2$. Moreover,
\begin{align*}
&\widetilde{II}_{2,2}(p) \lc \Big (\sum_{Q\in\fD_0}
\Big( \sum_{\substack{W\in \fW\\ W\subset Q}} \|b_{1,W}\|_p^p \Big)^{\frac 2p}
\Big)^{\frac 12}
\\&\lc \alpha_1 \Big( \sum_{Q\in\fD_0}\Big( \sum_{\substack{W\in \fW\\ W\subset Q}} |W|\Big)^{\frac 2p} \Big)^{\frac 12}
\lc \alpha_1 \Big( \sum_{Q\in\fD_0} \sum_{\substack{W\in \fW\\ W\subset Q}} |W| \Big)^{\frac 12} \lc \alpha_1|S|^{1/2}
\end{align*}
and hence $II_{2,2}\lc \widetilde{II}_{2,2}(p)II_{2,2}(q)\lc \alpha_1\alpha_2|S|.$
For $II_{2,3}$ we apply Proposition
\ref{prop:bilinear-var} with $\fQ=\fW_{>0}$ and with $\fQ'$ being the family of those $Q\in\fD_0$ which contain at least one cube in $\fW$.
Likewise for $II_{2,4}$ we use Proposition \ref{prop:bilinear-var} with the families $\fQ, \fQ'$ both consisting of those $Q\in \fD_0$ which contain at least one cube in $\fW$.

\subsubsection{Estimation of $II_3$}
Here we use Lemma \ref{lem:bilinear-error} and the assumptions on $p,q$ are irrelevant.
We can write $II_{3}\leq II_{3,1}+II_{3,2}$ with
\begin{align*}
II_{3,1}&= \Big| \sum_{\ell=0}^{L(S)-\nc} \sum_{j_1=1}^{\ell+\nc} \sum_{j_2>\ell+4\nc} \Gamma_{\fQ,\fQ'}^{\ell,j_1,j_2} (B_1,B_2) \Big|
\\
II_{3,2}&= \Big| \sum_{\ell=0}^{L(S)-\nc} \sum_{j_2>\ell+4\nc} \Gamma_{\fQ,\fQ'}^{\ell,0,j_2} (B_1,B_2) \Big|
\end{align*}
where
$\Gamma_{\fQ,\fQ'}^{\ell,j_1,j_2}$ is
as in \eqref{eq:Gammajsdef}.
Then $|\Gamma_{\fQ,\fQ'}^{\ell,j_1,j_2} (B_1,B_2)| \lc 2^{-j_2N} \|B_1\|_{\Xspace_{p,2}}\|B_2\|_{\Xspace_{q, 2}}$ and since we trivially have $\sum_{\ell\ge 0} \sum_{j_1=0}^{\ell+\nc} \sum_{j_2\ge \ell+4\nc} 2^{-j_2N} =O(1)$ we see that
\[
II_{3,1} \lc \Big (\sum_{\substack{Q\in \fW_{>0}} } 2^{-2L(Q) d(\frac 1p-\frac 12)} \|b_{1,Q}\|_p^2\Big)^{1/2} \Big (\sum_{\substack{Q'\in \fW_{>0}}} 2^{-2L(Q') d(\frac 1q-\frac 12)} \|b_{2,Q'}\|_q^2\Big)^{\frac 12}. \]
Arguing as for the term $II_2$, this immediately leads to $|II_{3,1}|\lc |S|\alpha_1\alpha_2$.
Similarly
\begin{align*}
|II_{3,2} |\lc \Big (\sum_{\substack{Q\in \fD_0} } \Big\|\sum_{\substack {W\in \fW\\W\subset Q}} b_{1,W} \Big\|_p^2\Big)^{1/2} \Big (\sum_{\substack{Q'\in \fW_{>0}}} 2^{-2L(Q') d(\frac 1q-\frac 12)} \|b_{2,Q'}\|_q^2\Big)^{\frac 12}
\end{align*}
and arguing as in the estimation of $II_{2,2}$ we obtain $II_{3,2} \lc |S|\alpha_1\alpha_2.$
\end{proof}

\subsection*{An open problem} It remains open whether for any $\la\in (0,\frac{d-1}{2})\setminus \{\frac{d-1}{2(d+1) }\}$ the sharp $\Sp(p_\la,q_\la) $ bound with $p_\la=\frac{2d}{d+1+2\la}$ and $\frac 1{q_\la}=\frac{d+1}{ (d-1)p_\la} -\frac{2}{d-1} $ holds
(and then also the sparse bounds at the top of the trapezoid $\trapez_d(\la)$). If in the analysis for the terms $II$ above we replace the Cauchy-Schwarz inequality by H\"older's inequality we see that we would need a sharp version of Lemma \ref{lem:endpoint-ST-summation} with
$\Xspace_{p_\la,q_\la'}^\fQ \to L^{q_\la'}$-boundedness for a disjoint family $\fQ$ of dyadic cubes, where $\Xspace_{p,r}^\fQ$ denotes the closed subspace of $\Xspace_{p,r}$ consisting of all $F=\{f_Q\}\in \Xspace_{p,r}$ such that $f_Q=0$ for $Q\notin \fQ$. The latter is analogous to verifying an endpoint version of $\VV(p,r)$ where one allows $\frac 1r=\frac{d+1}{d-1}(1-\frac 1p)$ and assumes that the $f_{j,R}$ are zero if $R\notin \fQ$. We do not know whether such endpoint inequalities hold for $r\neq 2$.

\section{\texorpdfstring{Consequences for weak type inequalities with $A_1$ weights}{Consequences for weak type inequalities with A1 weights}} \label{sec:weighted} We record some consequences of our sparse domination results on new weak type weighted inequalities for $\cR^\la_a$ when $\la<\frac{d-1}{2}$. Frey and Nieraeth \cite{FreyNieraeth} (extending earlier results in \cite{bernicot-frey-petermichl}) formulated general theorems about weak type weighted inequalities for operators in $\Sp(p,q)$, satisfying certain $A_1$ and reverse H\"older conditions.
Recall the definitions of the $A_1$, $\RH_\sigma$ characteristics for a nonnegative measurable function, i.e. a weight $w$:
\Be\begin{aligned} \label{eq:A1RHdef}
[w]_{A_1} &=\sup_B \Big(\intslash w(x)\ud x\Big) \Big(\mathrm{ess\, inf}\ci{x\in B} w(x)\Big)^{-1}
\\ \ [w]_{\RH_\sigma} &= \sup_B (\intslash_B w(x)^\sigma \ud x\Big)^{\frac 1\sigma} \Big(\intslash_B w(x)\ud x\Big)^{-1}
\end{aligned}\Ee
if $\sigma\in (1,\infty)$. The relevant class here is $A_1\cap \RH_\sigma$, for which both characteristics are finite; we recall that $w$ belongs to this class if and only if $w^\sigma\in A_1$ \cite{JohnsonNeugebauer91}.
By \cite[Theorem 1.4]{FreyNieraeth} operators in $\Sp(p,q)$ map $L^p(w)$ to $ L^{p,\infty}(w) $ provided that $w\in A_1\cap \RH_\sigma$ with $\sigma=(q'/p)'=\frac{q}{q+p-pq}$. More specifically,
\Be\label{eq:charact} \|T\|_{L^p(w)\to L^{p,\infty} (w)} \lc \|T\|_{\Sp(p,q)}
[w ^{(q'/p)'}]_{A_\infty}^{1+\frac 1p} [w]_{A_1}^{\frac 1p}
[w]_{\RH_{(q'/p)'}}^{\frac 1p},
\Ee
where $[v]_{A_\infty} := \sup_B (v(B))^{-1}\int_B M[v\bbone_B] (x)\ud x$ is Wilson's $A_\infty$-constant \cite{wilson87}, in which the supremum is taken over all balls.
For convergence results it is important to note that
the $A_1$, $A_\infty$ and $\RH_\sigma$ characteristics satisfy translation and dilation invariance properties, in the sense that the characteristics for $w(\cdot-h)$ and $t^dw(t\cdot)$ are the same as the corresponding characteristics for $w$.
In two dimensions Kesler and Lacey use \eqref{eq:charact} to obtain the weighted weak type $(p_\la,p_\la)$ inequalities for $\cR^\la_a$ when $w\in A_1\cap \RH_\sigma$ and $\sigma >\frac{4}{4-3p_\la}=\frac{3+2\la}{2\la} $. Using Theorems \ref{thm:new-sparse-bound-2D} and \ref{thm:STendpt} we can lower the reverse H\"older exponent by a factor of $4$.
\begin{cor}\label{cor:weighted2D}
Let $d=2$, $a>0$, $0<\la<1/2$, $p_\la= \frac{4}{3+2\la} $, $\sigma_\circ(\la) = \frac{3+2\la}{8\la} $. Assume that $\sigma>\sigma_\circ(\la)$ if $\la \in (0,\frac 12)\setminus \{\frac 16\}$ and $\sigma\ge \sigma_\circ(\la)=\frac 52$ if
$\la=\frac 16$.
Then for all $w\in A_1\cap \RH_\sigma$,
\[\cR^\la_{a,t}: L^{p_\la}(w,\bbR^2) \to L^{p_\la,\infty}(w,\bbR^2)\]
with operator norms uniform in $t$.
\end{cor}
Note that when $\la\to 1/2$ the reverse H\"older exponent tends to $1$ which is to be expected since no reverse H\"older condition is needed in Vargas' result \cite{Vargas96} for $p=1$, $\la=1/2$. Similar results can be formulated in higher dimensions for $\sigma_\circ(\lambda)=\frac{(d-1)(d+1+2\lambda)}{4d\lambda}$ and a partial range of $\lambda$, depending on the knowledge of sharp $L^p \to L^r$ for the Bochner--Riesz operator. In particular, in view of Remark \ref{rem:listofremarks}, this currently holds for $\frac{d-1}{2(d+1)} \leq \lambda < \frac{d-1}{2}$, which suffices to establish Theorem \ref{thm:A1}.

\subsection{Proof of Theorem \ref{thm:A1}}

We only prove the case $p>1$ since $p=1$ is Vargas' result \cite{Vargas96}. Let $\sigma_*=\sigma_\circ(\frac{d-1}{2(d+1)})=\frac{d+3}{2}>1$. It is well known \cite{CoifmanFefferman1974} that every $A_1$ weight belongs to $\RH_{\sigma(w)}$ for some $\sigma(w)>1$; without loss of generality we can assume $1<\sigma(w)< \sigma_*$. Let $p_1(w):=1+\frac{d-1}{d+1}(1-\frac{1}{\sigma(w)})$ and $1 < p < p_1(w)$. By the preceding discussion, we have that $\cR_{a,t}^{\lambda(p)}$ maps $L^p(w) \to L^{p,\infty}(w)$ provided
\begin{enumerate}[(i)]
\item $\sigma(w) > \sigma_\circ(\lambda(p))$;
\item $\frac{d-1}{2(d+1)}<\lambda(p) < \frac{d-1}{2}$.
\end{enumerate}
On the one hand, the condition $\sigma(w)> \sigma_\circ(\lambda(p))$ can be quickly seen to be equivalent to $\frac{d^2-1}{4d\sigma(w)-2(d-1)}< \lambda(p)$, which in turn is equivalent to the condition $p< \frac{d-1}{d+1}(\frac{2d}{d-1}-\frac{1}{\sigma(w)})=p_1(w)$, which holds by assumption.

On the other hand, since $\sigma_\circ(\lambda(p)) < \sigma(w) < \sigma_*=\sigma_\circ(\frac{d-1}{2(d+1)})$ and $\sigma_\circ(\lambda)$ decreases as a function of $\lambda$, we have $\lambda(p)>\frac{d-1}{2(d+1)}$. Moreover, since $p>1$ we have $\sigma(\lambda(p))>\sigma(\lambda(1))=\sigma(\frac{d-1}{2})$, which implies $\lambda(p)<\frac{d-1}{2}$, concluding the proof of (ii).

By the above-mentioned invariance properties and
\eqref{eq:charact}, the operator norms are uniform in $t$. Moreover, since the usual approximation of the identity results with $L^1$ kernels hold in $L^p(w)$ with $A_1$ weights one can use routine arguments to see that $\lim_{t\to\infty} \cR_{a,t}^{\la(p)} f= f$ in the $L^{p,\infty}(w)$ norm, for all $f\in L^p(w)$. \qed

\def\MR#1{}
\providecommand{\bysame}{\leavevmode\hbox to3em{\hrulefill}\thinspace}
\providecommand{\MR}{\relax\ifhmode\unskip\space\fi MR }
\providecommand{\MRhref}[2]{%
  \href{http://www.ams.org/mathscinet-getitem?mr=#1}{#2}
}
\providecommand{\href}[2]{#2}


\begin{thebibliography}{10}

\bibitem{BRS}
David Beltran, Joris Roos, and Andreas Seeger, \emph{Multi-scale sparse
  domination}, To appear in Mem. Amer. Math. Soc., arXiv:2009.000227, 2020.

\bibitem{BRS-endpoints}
\bysame, \emph{Endpoint sparse domination for classes of multiplier
  transformations}, Math. Z. \textbf{305} (2023), Article Number: 11, 52 pp.

\bibitem{BRS-expository}
\bysame, \emph{A note on endpoint {B}ochner--{R}iesz estimates}, Oberwolfach
  Preprints, 2023.

\bibitem{benea-bernicot-luque}
Cristina Benea, Fr\'ed\'eric Bernicot, and Teresa Luque, \emph{Sparse bilinear
  forms for {B}ochner {R}iesz multipliers and applications}, Trans. London
  Math. Soc. \textbf{4} (2017), no.~1, 110--128. \MR{3653057}

\bibitem{bernicot-frey-petermichl}
Fr\'ed\'eric Bernicot, Dorothee Frey, and Stefanie Petermichl, \emph{Sharp
  weighted norm estimates beyond {C}alder\'on-{Z}ygmund theory}, Anal. PDE
  \textbf{9} (2016), no.~5, 1079--1113. \MR{3531367}

\bibitem{CarlesonSjolin}
Lennart Carleson and Per Sj\"{o}lin, \emph{Oscillatory integrals and a
  multiplier problem for the disc}, Studia Math. \textbf{44} (1972), 287--299.
  \MR{361607}

\bibitem{ChoKimLeeShim2005}
Yonggeun Cho, Youngcheol Kim, Sanghyuk Lee, and Yongsun Shim, \emph{Sharp
  {$L^p$}-{$L^q$} estimates for {B}ochner-{R}iesz operators of negative index
  in {$\mathbb{R}^n$}, {$n\geq 3$}}, J. Funct. Anal. \textbf{218} (2005),
  no.~1, 150--167. \MR{2101218}

\bibitem{Christ-BR87}
Michael Christ, \emph{Weak type endpoint bounds for {B}ochner-{R}iesz
  multipliers}, Rev. Mat. Iberoamericana \textbf{3} (1987), no.~1, 25--31.
  \MR{1008443}

\bibitem{Christ-rough}
\bysame, \emph{Weak type {$(1,1)$} bounds for rough operators}, Ann. of Math.
  (2) \textbf{128} (1988), no.~1, 19--42. \MR{951506}

\bibitem{ChristSogge1}
Michael Christ and Christopher~D. Sogge, \emph{On the {$L^1$} behavior of
  eigenfunction expansions and singular integral operators}, Miniconferences on
  harmonic analysis and operator algebras ({C}anberra, 1987), Proc. Centre
  Math. Anal. Austral. Nat. Univ., vol.~16, Austral. Nat. Univ., Canberra,
  1988, pp.~29--50. \MR{953981}

\bibitem{ChristSogge2}
\bysame, \emph{The weak type {$L^1$} convergence of eigenfunction expansions
  for pseudodifferential operators}, Invent. Math. \textbf{94} (1988), no.~2,
  421--453. \MR{958838}

\bibitem{CoifmanFefferman1974}
R.~R. Coifman and C.~Fefferman, \emph{Weighted norm inequalities for maximal
  functions and singular integrals}, Studia Math. \textbf{51} (1974), 241--250.
  \MR{358205}

\bibitem{conde-alonso-etal}
Jos\'e~M. Conde-Alonso, Amalia Culiuc, Francesco Di~Plinio, and Yumeng Ou,
  \emph{A sparse domination principle for rough singular integrals}, Anal. PDE
  \textbf{10} (2017), no.~5, 1255--1284. \MR{3668591}

\bibitem{fefferman69}
Charles Fefferman, \emph{Inequalities for strongly singular convolution
  operators}, Acta Math. \textbf{124} (1970), 9--36. \MR{257819}

\bibitem{FeffermanBR73}
\bysame, \emph{A note on spherical summation multipliers}, Israel J. Math.
  \textbf{15} (1973), 44--52. \MR{320624}

\bibitem{FreyNieraeth}
Dorothee Frey and Zoe Nieraeth, \emph{Weak and strong type {$A_1-A_\infty$}
  estimates for sparsely dominated operators}, J. Geom. Anal. \textbf{29}
  (2019), no.~1, 247--282. \MR{3897012}

\bibitem{GuoWangZhang2022}
Shaoming Guo, Hong Wang, and Ruixiang Zhang, \emph{A dichotomy for
  {H}\"ormander-type oscillatory integral operators},
  {\verb+arxiv.org/abs/2210.05851+}.

\bibitem{GHI}
Larry Guth, Jonathan Hickman, and Marina Iliopoulou, \emph{Sharp estimates for
  oscillatory integral operators via polynomial partitioning}, Acta Math.
  \textbf{223} (2019), no.~2, 251--376. \MR{4047925}

\bibitem{HNS}
Yaryong Heo, F\"{e}dor Nazarov, and Andreas Seeger, \emph{Radial {F}ourier
  multipliers in high dimensions}, Acta Math. \textbf{206} (2011), no.~1,
  55--92. \MR{2784663}

\bibitem{hormander-ALPDO-1}
Lars H\"{o}rmander, \emph{The analysis of linear partial differential
  operators. {I}}, second ed., Grundlehren der mathematischen Wissenschaften
  [Fundamental Principles of Mathematical Sciences], vol. 256, Springer-Verlag,
  Berlin, 1990, Distribution theory and Fourier analysis. \MR{1065993}

\bibitem{JohnsonNeugebauer91}
R.~Johnson and C.~J. Neugebauer, \emph{Change of variable results for {$A_p$}-
  and reverse {H}\"{o}lder {${\rm RH}_r$}-classes}, Trans. Amer. Math. Soc.
  \textbf{328} (1991), no.~2, 639--666. \MR{1018575}

\bibitem{KeslerLacey}
Robert Kesler and Michael~T. Lacey, \emph{Sparse endpoint estimates for
  {B}ochner-{R}iesz multipliers on the plane}, Collect. Math. \textbf{69}
  (2018), no.~3, 427--435. \MR{3842215}

\bibitem{Kwon-Lee}
Yehyun Kwon and Sanghyuk Lee, \emph{Sharp resolvent estimates outside of the
  uniform boundedness range}, Comm. Math. Phys. \textbf{374} (2020), no.~3,
  1417--1467. \MR{4076079}

\bibitem{lacey-mena-reguera}
Michael~T. Lacey, Dario Mena, and Maria~Carmen Reguera, \emph{Sparse bounds for
  {B}ochner-{R}iesz multipliers}, J. Fourier Anal. Appl. \textbf{25} (2019),
  no.~2, 523--537. \MR{3917956}

\bibitem{LeeSanghyk2004}
Sanghyuk Lee, \emph{Improved bounds for {B}ochner-{R}iesz and maximal
  {B}ochner-{R}iesz operators}, Duke Math. J. \textbf{122} (2004), no.~1,
  205--232. \MR{2046812}

\bibitem{LeeSanghyk2018}
\bysame, \emph{Square function estimates for the {B}ochner-{R}iesz means},
  Anal. PDE \textbf{11} (2018), no.~6, 1535--1586. \MR{3803718}

\bibitem{lerner-nazarov}
Andrei~K. Lerner and Fedor Nazarov, \emph{Intuitive dyadic calculus: the
  basics}, Expo. Math. \textbf{37} (2019), no.~3, 225--265. \MR{4007575}

\bibitem{RdF1984}
Jos\'{e}~L. Rubio~de Francia, \emph{Factorization theory and {$A\sb{p}$}
  weights}, Amer. J. Math. \textbf{106} (1984), no.~3, 533--547. \MR{745140}

\bibitem{Seeger-Indiana}
Andreas Seeger, \emph{Endpoint estimates for multiplier transformations on
  compact manifolds}, Indiana Univ. Math. J. \textbf{40} (1991), no.~2,
  471--533. \MR{1119186}

\bibitem{seeger-BRwt}
\bysame, \emph{Endpoint inequalities for {B}ochner-{R}iesz multipliers in the
  plane}, Pacific J. Math. \textbf{174} (1996), no.~2, 543--553. \MR{1405600}

\bibitem{Ste70}
Elias~M. Stein, \emph{Singular integrals and differentiability properties of
  functions}, Princeton Mathematical Series, No. 30, Princeton University
  Press, Princeton, N.J., 1970. \MR{0290095 (44 \#7280)}

\bibitem{stein-weiss}
Elias~M. Stein and Guido Weiss, \emph{Introduction to {F}ourier analysis on
  {E}uclidean spaces}, Princeton University Press, Princeton, N.J., 1971,
  Princeton Mathematical Series, No. 32. \MR{0304972}

\bibitem{tao-STendpoint}
Terence Tao, \emph{Weak-type endpoint bounds for {R}iesz means}, Proc. Amer.
  Math. Soc. \textbf{124} (1996), no.~9, 2797--2805. \MR{1327048}

\bibitem{Tao-Indiana1998}
\bysame, \emph{The weak-type endpoint {B}ochner-{R}iesz conjecture and related
  topics}, Indiana Univ. Math. J. \textbf{47} (1998), no.~3, 1097--1124.
  \MR{1665753}

\bibitem{Vargas96}
Ana~M. Vargas, \emph{Weighted weak type {$(1,1)$} bounds for rough operators},
  J. London Math. Soc. (2) \textbf{54} (1996), no.~2, 297--310. \MR{1405057}

\bibitem{wilson87}
J.~Michael Wilson, \emph{Weighted inequalities for the dyadic square function
  without dyadic {$A_\infty$}}, Duke Math. J. \textbf{55} (1987), no.~1,
  19--50. \MR{883661}

\bibitem{Wu-JdA}
Shukun Wu, \emph{On the {B}ochner-{R}iesz operator in $\mathbb{R}^3$},
  ar{X}iv:2008.13043, to appear in Jour. d'Analyse.

\end{thebibliography}
\end{document}